\documentclass[12pt]{amsart}
\usepackage{amsmath}
\usepackage{amsthm}
\usepackage{amstext}
\usepackage{amssymb}
\usepackage{amsfonts}
\usepackage{young}
\usepackage{multicol}
\usepackage{mathrsfs}
\usepackage{stmaryrd}
\usepackage[mathscr]{euscript}
\usepackage{tikz}
\usetikzlibrary{matrix,arrows,decorations.pathmorphing} 
\usepackage{ytableau}

        \def\shuffle{\mathbin{\small{\sqcup}\!{\sqcup}}}

\newtheorem{theorem}{Theorem}[section]
\newtheorem*{theorem*}{Theorem}
\newtheorem{lemma}[theorem]{Lemma}

\newtheorem{corollary}[theorem]{Corollary}

\newtheorem{problem}[theorem]{Problem}

\theoremstyle{definition}
\newtheorem{example}[theorem]{Example}
\newtheorem{remark}[theorem]{Remark}

\def\r{{\mathfrak {row}}}
\def\i{{\mathfrak {lis}}}
\def\d{{\mathfrak {lds}}}

\DeclareMathAlphabet{\mathpzc}{OT1}{pzc}{m}{it}

\title{
$K$-theoretic Poirier-Reutenauer bialgebra
}

\numberwithin{equation}{section}

\begin{document}

\author{Rebecca Patrias}
\address{\hspace{-.3in} Department of Mathematics, University of Minnesota,
Minneapolis, MN 55414, USA}
\email{patri080@umn.edu}

\author{Pavlo Pylyavskyy}
\address{\hspace{-.3in} Department of Mathematics, University of Minnesota,
Minneapolis, MN 55414, USA}
\email{ppylyavs@umn.edu}

\date{\today
}

\thanks{
}

\subjclass{
Primary
05E99, % Algebraic combinatorics: None of the above, but in this
}

\keywords{}

\begin{abstract}
We use the $K$-Knuth equivalence of Buch and Samuel \cite{BS} to define a $K$-theoretic analogue of the Poirier-Reutenauer Hopf algebra. As an application, we 
rederive the $K$-theoretic Littlewood-Richardson rules of Thomas and Yong \cite{ThY2,ThY3} and of Buch and Samuel \cite{BS}.
\end{abstract}

\ \vspace{-.1in}

\maketitle

\section{Introduction}
\label{sec:intro}

\subsection{Poirier-Reutenauer Hopf algebra}

In \cite{PR}, Piorier and Reutenauer defined a Hopf algebra structure on the $\mathbb Z$-span of all standard Young tableaux, which was
was later studied, for example, in \cite{RT} and \cite{DHNT}. Implicitly, this algebra also appears in \cite{LTh}. Let us briefly recall the definition 
and illustrate it with few examples. 

A Young diagram or partition is a finite collection of boxes arranged in left-justified rows such that the lengths of the rows are 
weakly decreasing from top to bottom. We denote the shape of a Young diagram $\lambda$ by $(\lambda_1,\lambda_2,\ldots,\lambda_k)$, listing the lengths of each row, $\lambda_i$.
A Young tableau is a filling of the boxes of a Young diagram with positive integers
so that the fillings increase in rows and columns. 
We call a Young tableau a standard Young tableau if it is filled with positive integers $[k]$ for some $k$, where each 
integer appears exactly once. The tableau shown below is an example of a standard Young tableau of shape $(3,3,2)$.
\begin{center}
 \begin{ytableau}
  1 & 4 & 6 \\
2 & 5 & 7 \\
3 & 8
 \end{ytableau}
\end{center}

Given two partitions, $\lambda$ and $\mu$, such that $\mu \subset \lambda$, we define the skew diagram $\lambda/\mu$ to be the set of boxes of $\lambda$ that
do not belong to $\mu$. If the shape of $T$ is $\lambda/\mu$ where $\mu$ is the empty shape, we say that $T$ is of straight shape. 
The definitions of Young tableaux and standard Young tableaux extend naturally to skew diagrams. For example, the figure below shows 
a standard Young tableau of skew shape $(3,3,1)/(2,1)$.

\begin{center}
 \begin{ytableau}
  \none & \none & 3 \\
\none & 1 & 4 \\
2
 \end{ytableau}
\end{center}

Given a possibly skew Young tableau $T$, its row reading word, $\r(T)$, is obtained by reading the entries in the rows of $T$ from left to right 
starting with the bottom row and ending with the top row. For the first standard Young tableau shown above, $\r(T)=38257146$, and for the standard Young tableau of skew shape, $\r(T)=2143$.

Next, consider words with distinct letters on some ordered alphabet $A$. We have the following \textit{Knuth relations}: 
%\begin{eqnarray}
% xzy &\approx& zxy \text{  if  } x\leq y < z, \text{ and} \\
%yxz &\approx& yzx \text{  if  } x<y\leq z.
%\end{eqnarray}
\begin{center}
\begin{tabular}{l c l}
% & $pp \equiv p$ & for all $p$ \vspace{.1in}\\
% & $pqp \equiv qpq$ & for all $p$ and $q$\vspace{.1in} \\
 & $pqs \approx qps$ and $sqp \approx spq$ & whenever $p<s<q$.
 \end{tabular}
\end{center}

Given two words, $w_1$ and $w_2$, we say that they are Knuth equivalent, denoted $w_1\approx w_2$, if $w_2$ can be obtained from $w_1$ by a finite sequence of 
Knuth relations. For example, $52143 \approx 25143$ because $$52143 \approx 52413 \approx 25413 \approx 25143.$$ If $T_1$ and $T_2$ are two tableaux, we say that 
$T_1 \approx T_2$ if $\r(T_1)\approx \r(T_2)$. For example, 
\begin{center}
 $T_1=$
\begin{ytableau}
 1 & 2 \\
3
\end{ytableau}\hspace{.5in}
$\approx$
\hspace{.5in}$T_2=$
\begin{ytableau}
 \none & \none & 2 \\
\none & 3 \\
1
\end{ytableau}\hspace{.1in}.

\end{center}

From Theorem 5.2.5 of \cite{LTh}, any word with letters exactly $[k]$ is Knuth equivalent to $\r(T)$ for a 
unique standard Young tableau $T$ of straight shape. This unique standard Young tableau may be obtained via RSK insertion of the 
word (see \cite{EC2}). 
For example, $52143\approx \r(T)$ for 
\begin{center}
$T=$
 \begin{ytableau}
  1 & 3 \\
2 & 4 \\
5
 \end{ytableau}\hspace{.05in}.
\end{center}

For a standard Young tableau $T$, let ${\bf T}=\sum_{w \approx \r(T)} w$. In other words, ${\bf T}$ is the sum of words 
that are Knuth equivalent to $\r(T)$. Let $PR$ be the $\mathbb{R}$-vector space generated by the set of $\bf T$ for all standard Young tableaux.

Following \cite{PR}, we next describe a bialgebra structure on $PR$. Start with two words, $w_1$ and $w_2$, in $PR$, where $w_1$ has letters 
exactly $[n]$ for some positive integer $n$. Define $w_2[n]$ to be the word obtained by adding $n$ to each letter of $w_2$. Now define the product $w_1 \ast w_2$ 
to be $w_1 \shuffle w_2[n]$, the shuffle product of $w_1$ and $w_2[n]$. For example, $12 \ast 1=12\shuffle 3=123 + 132 + 312$.

For a word $w$ without repeated letters, define $st(w)$ to be the unique word on $\{1,2,\ldots,|w|\}$ obtained by applying 
the unique order-preserving injective mapping from the letters of $w$ onto $\{1,2,\ldots,|w|\}$ to the letters of $w$. For example, 
$st(1426)=1324$. Then define the coproduct on $PR$ by defining $$\Delta(w)=\sum st(u)\otimes st(v),$$ where the sum is over 
all words $u$ and $v$ such that $w$ is the concatenation of $u$ and $v$. For example, $\Delta(312)=\emptyset \otimes 312 + 1\otimes 12 + 21 \otimes 1 + 312 \otimes \emptyset$, 
where $\emptyset$ denotes the empty word. As shown in \cite{PR}, the vector space $PR$, where we extend product $\ast$ and coproduct $\Delta$ by linearity,
forms a bialgebra.

\subsection{Two versions of the Littlewood-Richardson rule}
While being interesting in its own right, the
Poirier-Reutenauer Hopf algebra allows us to obtain a version of the Littlewood-Richardson rule for the cohomology rings of Grassmannians. In other words, it 
yields an explicitly positive description for the structure constants of the cohomology ring in the basis of {\it {Schubert classes}}. It is well-known that 
the Schubert classes can be represented by Schur functions of partitions that fit inside a rectangle. Thus, an essentially equivalent formulation of the problem is to describe 
structure constants of the ring of symmetric functions in terms of the basis of Schur functions. We refer the reader to \cite{M} for a great introduction to 
the subject. 

To see how the Poirier-Reutenauer Hopf algebra helps, let us state the following theorems.

 \begin{theorem} \cite[Theorem 5.4.3]{LTh} \label{thm:insold}
Let $T_1$ and $T_2$ be two standard Young tableaux. 
Then we have 
$${\bf T_1} \ast {\bf T_2} =
\displaystyle\sum_{T\in T(T_1 \shuffle T_2)} {\bf T},$$
where $T(T_1 \shuffle T_2)$ is the set of standard tableaux $T$ such that $T|_{[n]} = T_1$ and $T|_{[n+1,n+m]} \approx T_2.$
\end{theorem}

Given a tableau $T$, define $\overline{T}$ to be the tableau of the same shape as $T$ with reading word $st(\r(T))$.
The following theorem is analogous to Theorem \ref{thm:insold} and is not hard to prove using the methods of \cite{LTh}.

\begin{theorem} \label{thm:cohecktoinsold}
 Let $S$ be a standard Young tableau. We have 
$$\Delta({\bf S}) =
\displaystyle\sum_{(T',T'') \in T(S)} {\bf \overline T'} \otimes {\bf \overline T''},$$
where $T(S)$ is the set of pairs of tableaux $T',T''$ such that $\r(T') \r(T'') \approx \r(S)$.
\end{theorem}

Let $\Lambda$ denote the ring of symmetric functions. Denote by $s_{\lambda}$ its basis of Schur functions, mentioned above. See  for example \cite{EC2} for details.
Then $\Lambda$ has a bialgebra structure, see \cite{Zel} for details. 

We are interested in a combinatorial rule for the coefficients $c_{\lambda,\mu}^{\nu}$ in the decompositions 
$$s_{\lambda} s_{\mu} = \sum_{\nu} c_{\lambda,\mu}^{\nu} s_{\nu}.$$

Define $\psi: PR \longrightarrow \Lambda$ by 
$$\psi({\bf T}) = s_{\lambda(T)},$$
where $\lambda(T)$ denotes the shape of $T$.

\begin{theorem} \cite[Theorem 5.4.5]{LTh} \label{thm:psi}
 The map $\psi$ is a bialgebra morphism. 
\end{theorem}

Applying $\psi$ to the equalities in Theorem \ref{thm:insold} and Theorem \ref{thm:cohecktoinsold}, we obtain the following two versions of the Littlewood-Richardson rule. 

\begin{corollary} \cite[Theorem A1.3.1]{EC2} \label{cor:LRold}
 Let $T$ be a standard Young tableau of shape $\mu$. Then the coefficient $c_{\lambda, \mu}^{\nu}$ 
is equal to the number of standard Young tableaux $R$
of skew shape $\nu/\lambda$ such that $\r(R) \approx \r(T)$.
\end{corollary}

\begin{corollary} \cite[Theorem 5.4.5]{LTh} \label{cor:skewLRold}
 Let $S$ be a standard Young tableau of shape $\nu$. Then the coefficient $c_{\lambda, \mu}^{\nu}$ in the decomposition 
is equal to the number of standard Young tableaux $R$
of skew shape $\lambda \oplus \mu$ such that $\r(R) \approx S$.
\end{corollary}

%\subsection{Hecke insertion and restricted Hecke equivalence}

\subsection{$K$-theoretic Poirier-Reutenauer bialgebra and Littlewood-Richardson rule}

The combinatorics of the $K$-theory of Grassmannians has been developed in \cite{FG, FK, LS}. In \cite{B} Buch gave an explicit description
of the {\it {stable Grothendieck polynomials}}, which represent Schubert classes in the $K$-theory ring. Such a description was already implicit in \cite{FG}. 
Then Buch proceeded to give a Littlewood-Richardson rule,
which describes the structure constants of the ring with respect to the basis of those classes. An alternative description of those structure constants was 
obtained by Thomas and Yong in \cite{ThY2,ThY3}. 

In \cite{BKSTY}, a natural analogue of Knuth insertion called {\it {Hecke insertion}} is defined. A result of such insertion is an {\it {increasing tableau}}, which is a 
natural analogue of a standard Young tableau. 

A question arises then: can one use Hecke insertion 
to define a $K$-theoretic analogue of the Poirier-Reutenauer Hopf algebra? Can one then proceed to obtain a version of the Littlewood-Richardson rule analogous to 
Corollary \ref{cor:LRold} and Corollary \ref{cor:skewLRold}? It turns out the answer is yes, although there are additional obstacles to overcome. 
This is the goal of this paper.

%Our construction is described in detail in the future sections. Here we give a brief overview with emphasis on the nature of obstacles which appear 
%and on methods we use to overcome them.

It turns out that there is no {\it {local}} way to describe equivalence between words that Hecke insert into the same tableaux. This was, of course, already known 
in \cite{BKSTY}. The consequence is that the verbatim definition of the Poirier-Reutenauer bialgebra simply does not work. If for an increasing tableau $T$ we define 
$${\bf T} = \sum_{P(w) = T} w,$$ where the sum is over all words that Hecke insert into $T$,
the resulting sums are not closed under the natural product and coproduct, see Remark \ref{rem:prod} and Remark \ref{rem:coprod}.

We use instead classes defined by the 
\textit{$K$-Knuth equivalence relation} of \cite{BS}, a combination of the Hecke equivalence of \cite{B} and Knuth equivalence. The relation is defined 
by the following three local rules:

\begin{center}
\begin{tabular}{l c l}
 & $pp \equiv p$ & for all $p$ \vspace{.1in}\\
 & $pqp \equiv qpq$ & for all $p$ and $q$\vspace{.1in} \\
 & $pqs \equiv qps$ and $sqp \equiv spq$ & whenever $p<s<q$.
 \end{tabular}
\end{center}

%It provides enough equivalences to have the vector space closed under the product and coproduct but not too many to prevent us from deriving a version of the Littlewood-Richardson rule
%by using the resulting bialgebra. 

It is important to note that the $K$-Knuth classes combine some classes of increasing tableaux, as seen in \cite{BS}. In other words, there are $K$-Knuth equivalence classes of words that have more than one
corresponding tableau. For example,
the $K$-Knuth equivalence class of $3124$ contains {\it {two}} increasing tableaux, shown below.
\begin{center}
 \begin{ytableau}
  1 & 2 & 4 \\
3 & 4
 \end{ytableau}\hspace{1in}
\begin{ytableau}
  1 & 2 & 4 \\
3 
 \end{ytableau}\hspace{.1in}.
\end{center}
We invite the reader to verify that the row reading words of those tableaux can be indeed connected to each other by $K$-Knuth equivalences. 

In order to get a working version of the Littlewood-Richardson rule, such tableaux need to be avoided. We use the notion of a \textit{unique rectification target} of Buch 
and Samuel, see \cite{BS}, which are increasing tableaux with the property of being the only increasing tableau in their $K$-Knuth equivalence class. We will refer to a unique rectification target 
as a URT.

Finally, armed with this notion of unique rectification targets, we can state and prove the following versions of the Littlewood-Richardson rule, similar to those of 
Corollary \ref{cor:LRold} and Corollary \ref{cor:skewLRold}. The first was proven previously in \cite[Corollary 3.19]{BS} and in less generality in \cite[Theorem 1.2]{ThY2} 
using a K-theoretic analogue of jeu de taquin. 
%The connection between jeu de taquin formulation and Hecke insertion formulation is established by \cite[Theorem 6.2]{BS}.

\begin{theorem*} [Theorem \ref{thm:LR}]
 Let $T$ be a URT of shape $\mu$. Then the coefficient $c_{\lambda, \mu}^{\nu}$ in the decomposition 
$$G_{\lambda} G_{\mu} = \sum_{\nu} (-1)^{|\nu|-|\lambda|-|\mu|} c_{\lambda, \mu}^{\nu} G_{\nu}$$ is equal to the number of increasing tableaux $R$
of skew shape $\nu/\lambda$ such that $P(\r(R))=T$.
\end{theorem*}

%This closely resembles and could be related to \cite[Theorem 1.4]{ThY3}.
%Indeed, by CITE JDT=HECKEINS and \cite[Theorem 1.2]{ThY3}, 
%the result of a jeu de taquin procedure (with arbitrary rectification order) in this case would be the same as the result of Hecke insertion.

While we obtain the next result only for unique rectification targets, \cite[Theorem 1.4]{ThY3} proves it for arbitrary increasing tableaux.
\begin{theorem*} [Theorem \ref{thm:skewLR}]
 Let $T_0$ be a URT of shape $\nu$. Then the coefficient $d_{\lambda, \mu}^{\nu}$ in the decomposition 
$$\Delta(G_{\nu}) = \sum_{\lambda, \mu} (-1)^{|\nu|-|\lambda|-|\mu|} d_{\lambda, \mu}^{\nu} G_{\lambda} \otimes G_{\mu}$$ is equal to the number of increasing tableaux $R$
of skew shape $\lambda \oplus \mu$ such that $P(\r(R))=T_0$.
\end{theorem*}

%This closely resembles and could be related to \cite[Theorem 1.4]{ThY2}.
%For a special choice of $S$, this recovers the rule of \cite[Theorem 1.4]{ThY2}. 
%Indeed, according to CITE JDT=HECKEINS and \cite[Theorem 1.2]{ThY2}, 
%the result of a jeu de taquin procedure (with arbitrary rectification order) in this case would be the same as the result of Hecke insertion.
%While we obtain this result for arbitrary solitary $S$, \cite[Theorem 1.4]{ThY2} obtains it only for a specific choice of $S$. 
\begin{remark}
Let us make the relationship between the two previous theorems, the two theorems of Thomas and Yong: Theorem \cite[Theorem 1.4]{ThY3} and \cite[Theorem 1.4]{ThY2}, and the result of Buch and Samuel 
\cite[Corollary 3.19]{BS} clear. Their theorems are stated in terms of $K$-theoretic jeu de taquin, which is introduced by Thomas and Yong. In \cite{BS}, 
Buch and Samuel prove that $K$-Knuth equivalence is equivalent to $K$-theoretic jeu de taquin equivalence in \cite[Theorem 6.2]{BS}, thus explaining the connection.

Therefore, our Theorem \ref{thm:LR} is a corollary of \cite[Theorem 1.4]{ThY3}, where our theorem is more specialized since we require $T$ to be a URT. On the other 
hand, Theorem \ref{thm:skewLR} is the same as \cite[Corollary 3.10]{BS}, which both generalize \cite[Theorem 1.4]{ThY2}, as we allow $S$ to be an arbitrary URT rather than fixing a particular 
({\it {superstandard}}) choice. 

%At the moment however, the only result we are aware of of the nature required is \cite[Theorem 4.2]{ThY}, and it does not suffice to establish the relation.
%In fact, our Example \ref{ex:ThYfail} suggests that any connection between the two approaches cannot be too simplistic. 
%Indeed, it shows that our theorem does require solitarity assumption, while \cite[Theorem 1.4]{ThY3} does not. 
\end{remark}

In our proof of an analogue of Theorem \ref{thm:psi}, it is more natural to work with the {\it {weak set-valued tableaux}}
defined in \cite{LP} than with the set-valued tableaux of Buch \cite{B}. However, as we show in Corollary \ref{cor:coincide}, the two languages are equivalent. 

%\begin{remark}
% It would be very interesting to understand the connection between our work and that of Thomas and Yong, \cite{ThY2, ThY3}. Both use increasing tableaux, but in 
%our case we use Hecke insertion, while Thomas and Yong use a version of {\it {jeu de taquin}}.
%\end{remark}

\subsection{Plan of the paper and acknowledgements}
In Section 2, we describe Hecke insertion and reverse Hecke insertion. We define the insertion tableau, $P(w)$, and the recording tableau, $Q(w)$, for a word $w$.
We review several relevant results regarding Hecke insertion. We then recall (from \cite{BS}) the \textit{$K$-Knuth equivalence} of finite words 
on the alphabet $\{1,2,3,\ldots\}$ and discuss certain characteristics of this equivalence.

In Section 3,  we define $[[h]]$ to be the sum of all words in the Hecke equivalence class of a word $h$. We define a vector space, $KPR$, 
spanned by all such sums. We introduce a bialgebra structure on $KPR$ and show that $KPR$ has no antipode for this bialgebra structure. Thus, we 
obtain the $K$-theoretic Poirier-Reutenauer bialgebra. 

In Section 4, we recall from \cite{BS} the notion of a \textit{unique rectification target} (URT), a tableau that is the unique tableau in its $K$-Knuth equivalence class. 
We rephrase the product and coproduct formulas from the previous section 
for $K$-Knuth equivalence classes that correspond to URTs.

In Section 5, we draw a connection between the material in the previous sections and the ring of symmetric functions. We define 
the \textit{stable Grothendieck polynomials}, $G_\lambda$, as in \cite{B} by using set-valued tableaux and discuss their structure constants. 
We then use \textit{weak set-valued tableaux} to define \textit{weak stable Grothendieck polynomials}, $J_\lambda$. We show that the 
bialgebra structure constants of the $G_\lambda$ and the $J_\lambda$ coincide up to a sign. Using the fundamental 
quasisymmetric functions, we define a bialgebra morphism, $\phi$, with the property that $\phi([[h]])$ can be written as a sum of weak 
stable Grothendieck polynomials. 

In Section 6, we use the bialgebra morphism from Section 5 to state and prove a Littlewood-Richardson rule for the product and coproduct of the stable Grothendieck polynomials. 

\

We are grateful to Oliver Pechenik, Alex Yong and Thomas Lam for helpful comments on the first draft of the paper. 

\section{Hecke insertion and the $K$-Knuth monoid}

\subsection{Hecke insertion}
 
An \textit{increasing tableau} is a filling of a Young diagram with positive integers such that the entries in rows are strictly
increasing from left to right and the entries in columns are strictly increasing
from top to bottom. 

\begin{example}
 The tableau shown on the left is an increasing tableau. The tableau on the right is not an increasing tableau
because the entries in the first row are not strictly increasing.
\begin{center}
\begin{ytableau}
 1 & 2 & 4 & 5 \\
2 & 3 & 5 & 7 \\
6 & 7 \\
8
\end{ytableau}\hspace{1in}
\begin{ytableau}
 1 & 2 & 2 & 4 \\
3 & 4 \\
5
\end{ytableau}
\end{center}
\end{example}

\begin{lemma} \label{lem:fin}
 There are only finitely many increasing tableaux filled with a given finite alphabet.
\end{lemma}

\begin{proof}
 If the alphabet used has $n$ letters, each row and each column cannot be longer than $n$. 
\end{proof}

Of particular importance in what follows will be increasing tableaux on alphabets consisting of the first several
positive integers, i.e. on $[k] = \{1,2,\ldots,k\}$. We call such increasing tableaux {\it {initial}}.

We follow \cite{BKSTY} to give a description of Hecke (row) insertion of a positive integer $x$ 
into an increasing tableau $Y$ resulting in an increasing tableau $Z$.
The shape of $Z$ is obtained from the shape of $Y$ by adding at most one box. If a box is added in position $(i,j)$, then
we set $c=(i,j)$. 
In the case where no box is added, then $c=(i,j)$, where $(i,j)$ is a special corner indicating where the insertion
process terminated. We will use a parameter $\alpha\in\{0,1\}$ to keep track
of whether or not a box is added to $Y$ after inserting $x$ by setting
$\alpha=0$ if $c\in Y$ and $\alpha=1$ if $c\notin Y$. 
We use the notation $Z=(Y {\overset{H}{\longleftarrow}} x)$ to denote the resulting tableau, and we denote the 
outcome of the insertion by $(Z,c,\alpha)$.

We now describe how to insert $x$ into increasing tableau $Y$ by describing
how to insert $x$ into $R$, a row of $Y$. This insertion may 
modify the row and may produce an output integer, 
which we will insert into the next row. To begin the insertion process, insert $x$ into the first row of $Y$.  
The process stops when there is no output integer. 
The rules for insertion of $x$ into $R$ are as follows:
\begin{itemize}
\item[(H1)] If $x$ is weakly larger than all integers in $R$ and adjoining $x$ to the end of row $R$ 
results in an increasing tableau, then $Z$ is the resulting tableau and $c$ is the new
corner where $x$ was added.

\item[(H2)] If $x$ is weakly larger than all integers in $R$ and adjoining $x$ to the end of row $R$ 
does not result in an increasing tableau, then $Z=Y$, and $c$ is the 
box at the bottom of the column of $Z$ containing the rightmost box of the row $R$. 
\end{itemize}
For the next two rules, assume $R$ contains boxes strictly larger than $x$, and let $y$ be the smallest such box.
\begin{itemize}
\item[(H3)] If replacing $y$ with $x$ results in an increasing tableau, then replace $y$ with $x$. 
In this case, $y$ is the output integer to be inserted into the next row
\item[(H4)] If replacing $y$ with $x$ does not result in an increasing tableau, then do not change row $R$. 
In this case, $y$ is the output integer to be inserted into the next row.
\end{itemize}

\begin{example}\text{ }\\
\begin{center}
\begin{ytableau}
  $1$ & $2$ & $3$ & $5$ \\
$2$ & $3$ & $4$ & $6$ \\
$6$ \\
$7$ 
 \end{ytableau}\text{ }${\overset{H}{\longleftarrow}} 3$\text{ } $=$ \text{ }
\begin{ytableau}
  $1$ & $2$ & $3$ & $5$ \\
$2$ & $3$ & $4$ & $6$ \\
$6$ \\
$7$
 \end{ytableau} \end{center}
We use rule (H4) in the first row to obtain output integer $5$. Notice that the $5$ cannot replace the $6$ 
in the second row since it would be directly below the $5$ in the first row. 
Thus we use (H4) again and get output integer $6$. Since we cannot add this $6$ to the end of the third row, 
we use (H2) and get $c=(1,4)$. Notice that the shape did not
change in this insertion, so $\alpha=0$.
\end{example}
\begin{example}  \label{ex:ins2} \text{ }\\
\begin{center}
\begin{ytableau} 
 $2$ & $4$ & $6$ \\
$3$ & $6$ & $8$ \\
$7$
\end{ytableau}\text{ }${\overset{H}{\longleftarrow}} 5$\text{ } $=$ \text{ }
\begin{ytableau} 
 $2$ & $4$ & $5$ \\
$3$ & $6$ & $8$ \\
$7$ & $8$
\end{ytableau}
\end{center}
The integer $5$ bumps the $6$ from the first row using (H3). The $6$ is inserted into the second row, 
which already contains a $6$. Using (H4), the second row remains unchanged
and we insert $8$ into the third row. Since $8$ is larger than everything in the third row, we use (H1) 
to adjoin it to the end of the row. Thus $\alpha=1$.
\end{example}

In \cite{BKSTY}, Buch, Kresch, Shimozono, Tamvakis, and Yong give the following algorithm for reverse Hecke 
insertion starting with the triple $(Z,c,\alpha)$ as described above and ending with a pair $(Y,x)$ consisting of an 
increasing tableau and a postive integer. 

\begin{itemize}
\item[(rH1)] If $y$ is the cell in square $c$ of $Z$ and $\alpha=1$, 
then remove $y$ and reverse insert $y$ into the row above.

\item[(rH2)] If $\alpha=0$, do not remove $y$, 
but still reverse insert it into the row above.
\end{itemize}
In the row above, let $x$ be the largest integer 
such that $x<y$. 

\begin{itemize}
\item[(rH3)] If replacing $x$ with $y$ results in an 
increasing tableau, then we replace $x$ with $y$ and reverse insert $x$ into the row above.

\item[(rH4)] If replacing
$x$ with $y$ does not result in an increasing tableau, leave the row unchanged and reverse insert $x$ into the row
above.

\item[(rH5)] If $R$ is the first row of the tableau, the final output consists of $x$ and the modified tableau.
\end{itemize}

\begin{theorem} \cite[Theorem 4]{BKSTY}\label{thm:insbijection}
 Hecke insertion $(Y,x)\mapsto (Z,c,\alpha)$ and reverse Hecke insertion $(Z,c,\alpha)\mapsto (Y,x)$ define
mutually inverse bijections between the set of pairs consisting of an increasing tableau and a positive integer and the
set of triples consisting of an increasing tableau, a corner cell of the increasing tableau, and $\alpha\in\{0,1\}$.
\end{theorem}

Buch, Kresch, Shimozono, Tamvakis, and Yong prove the following lemma about Hecke insertion, which
will be useful later.

\begin{lemma} \cite[Lemma 2]{BKSTY}\label{lem:below}
 Let $Y$ be an increasing tableau and $x_1,x_2$ be two positive integers. Suppose that Hecke insertion 
of $x_1$ into $Y$ results in $(Z,c_1)$ and Hecke insertion
of $x_2$ into $Z$ results in $(T,c_2)$. Then $c_2$ is strictly below $c_1$ if and only if $x_1>x_2$.
\end{lemma}

Define the {\it {row reading}} word of an increasing tableau $T$, $\r(T)$, to be its content read left to right in each row, starting from the bottom row 
and ending with the top row. 

\begin{example}
 The second tableau in Example \ref{ex:ins2} has the reading word $78368245$.
\end{example}

Suppose $w=w_1w_2\ldots w_n$ is a word. Its \textit{insertion tableau} is $$P(w)=(\ldots((\emptyset {\overset{H}{\longleftarrow}} w_1){\overset{H}{\longleftarrow}} w_2)\ldots {\overset{H}{\longleftarrow}} w_n).$$

We shall also need the following two lemmas.

\begin{lemma} \label{lem:restrict}
 If $P(w)=T$ then $P(w)|_{[k]} = P(w|_{[k]})=T|_{[k]}$. 
\end{lemma}

\begin{proof}
 This follows from the insertion rules; letters greater than $k$ never affect letters in $[k]$.
\end{proof}

\begin{lemma}\label{lem:insrowword}
 For any tableau $T$, $P(\r(T))=T$.
\end{lemma}
\begin{proof}
It is easy to see that when each next row is inserted, it pushes down the previous rows.  
\end{proof}

\subsection{Recording tableaux} \label{sec:rec}

A \textit{set-valued tableau} $T$ of shape $\lambda$ is a filling of the boxes
with finite, non-empty subsets of positive integers so that
\begin{enumerate}
 \item the smallest number in each box is greater than or equal to the largest
number in the box directly to the left of it (if that box is present), and
\item the smallest number in each box is strictly greater than the largest number 
in the box directly above it (if that box is present).
\end{enumerate}

Given a word $h=h_1h_2\ldots h_l$, we can associate a pair of tableaux $(P(h), Q(h))$, where $P(h)$ is the insertion tableau described previously 
and $Q(h)$ is a set-valued tableau called the \textit{recording tableau} obtained as follows. Start with $Q(\emptyset)=\emptyset$. At each step of the insertion of $h$, let $Q(h_1\ldots h_k)$ be 
obtained from $Q(h_1\ldots h_{k-1})$ by labeling the special corner, $c$, in the insertion of $h_k$ into $P(h_1\ldots h_{k-1})$ with the positive integer $k$. Then 
$Q(h)=Q(h_1 h_2\ldots h_l)$ is the resulting strictly increasing set-valued tableau.

\begin{example}
 Let $h$ be $15133$. We obtain $(P(h),Q(h))$ with the following sequence, where in column $k$, $Q(h_1 \ldots h_k)$ is shown below $P(h_1 \ldots h_k)$. 
\begin{center}
\begin{ytableau}
 1
\end{ytableau}\hspace{.3in}
\begin{ytableau}
 1 & 5
\end{ytableau}\hspace{.3in}
\begin{ytableau}
 1 & 5 \\
5
\end{ytableau}\hspace{.3in}
\begin{ytableau}
 1 & 3 \\
5
\end{ytableau}\hspace{.3in}
\begin{ytableau}
 1 & 3 \\
5
\end{ytableau} 
$=P(h)$
\end{center}
\vspace{.2in}
\begin{center}
\begin{ytableau}
 1
\end{ytableau}\hspace{.3in}
\begin{ytableau}
 1 & 2
\end{ytableau}\hspace{.3in}
\begin{ytableau}
 1 & 2 \\
3
\end{ytableau} \hspace{.3in}
\begin{ytableau}
 1 & 2 \\
34
\end{ytableau}\hspace{.3in}
\begin{ytableau}
 1 & 25 \\
34
\end{ytableau}
$=Q(h)$
\end{center}
\end{example}

Call a word $h$ {\it {initial}} if the letters appearing in it are exactly the numbers in $[k]$ for some positive integer $k$. 

\begin{example}
 The word $13422$ is initial since the letters appearing in it are the numbers from $1$ to $4$. On the other hand, the word $1422$ is not initial 
because the letters appearing in it are $1$, $2$, and $4$ and do not form a set $[k]$ for any $k$. 
\end{example}

\begin{theorem} \label{thm:bij}
The map sending $h=h_1h_2\cdots h_n$ to $(P(h),Q(h))$ is a bijection between words and ordered pairs of tableaux of the same shape $(P,Q)$, where $P$ is an increasing tableau 
and $Q$ is a set-valued tableau with entries $\{1,2,\ldots,n\}$. It is also a bijection if there is an extra condition of being initial imposed both on $h$ and $P$.
\end{theorem}
\begin{proof}
It is clear from the definition of $Q(h)$ that $P(h)$ and $Q(h)$ have the same shape, and it is clear 
from the insertion algorithm that $P(h)$ is an increasing tableau and $Q(h)$ is an increasing set-valued tableau. Thus, we must show that given $(P,Q)$, one can uniquely recover $h$. 

To recover $h$, perform reverse Hecke insertion in $P$ multiple times as follows. Let $l$ be the the largest entry in $Q$ and call its cell $c(l)$. If $l$ is the only entry in $c(l)$ inside in $Q$, 
perform reverse Hecke insertion with the triple $(P,c(l),\alpha=1)$. 
If the $l$ is not the only entry in its cell in $Q$, perform reverse Hecke insertion with the triple $(P,c(l),\alpha=0)$. This reverse Hecke insertion will end with output $(P_2,x_l)$.  
Set $Q_2=Q-\{l\}$, and follow the same procedure described above replacing $Q$ with $Q_2$ and $P$ with $P_2$. 
The reverse insertion will end with output $(P_3,x_{l-1})$. Set $Q_3=Q_2-\{l-1\}$. Continue this process until the output tableau is empty.
By Theorem \ref{thm:insbijection}, $h=x_1x_2\ldots x_l$, $P(h)=P$, and $Q(h)=Q$.
\end{proof}

\begin{example}
Let's start with the pair $(P,Q)$ from the previous example and recover $h$.
\begin{center}
$P=$
\begin{ytableau}
  1 & 3 \\
5
\end{ytableau}\hspace{1in}
$Q=$
\begin{ytableau}
  1 & 25 \\
34
\end{ytableau}
\end{center}
We first notice the largest entry of $Q$ is in cell $(1,2)$ and is not the smallest entry in cell $(1,2)$, so we perform the 
reverse Hecke insertion determined by the triple $(P,(1,2),0)$. The output of this reverse insertion is $(P_2,3)$, so $h_5=3$.
\begin{center}
 $P_2=$
\begin{ytableau}
  1 & 3 \\
5
\end{ytableau}\hspace{1in}
$Q_2=$
\begin{ytableau}
  1 & 2 \\
34
\end{ytableau}
\end{center}
The largest entry in $Q_2$ is in cell $(2,1)$ and is not the smallest entry in cell $(2,1)$, so we perform the reverse Hecke insertion determined by $(P_2,(2,1),0)$ and obtain output
$(P_3,3)$. Thus $h_4=3$.
\begin{center}
 $P_3=$
\begin{ytableau}
  1 & 5 \\
5
\end{ytableau}\hspace{1in}
$Q_3=$
\begin{ytableau}
  1 & 2 \\
3
\end{ytableau}
\end{center}
The largest entry in $Q_3$ is in cell $(2,1)$ and is the smallest entry in its cell. We perform reverse insertion $(P_3,(2,1),1)$, obtain output $(P_4,1)$, and set $h_3=1$.
\begin{center}
$P_4=$
\begin{ytableau}
  1 & 5 \\
\end{ytableau}\hspace{1in}
$Q_4=$
\begin{ytableau}
  1 & 2 \\
\end{ytableau}
\end{center}
In the last two steps, we recover $h_2=5$ and $h_1=1$.

\end{example}

\subsection{$K$-Knuth equivalence}

We next introduce the {\it {$K$-Knuth monoid}} of \cite{BS} as the quotient of the free monoid of all finite words on the alphabet $\{1,2,3,\ldots\}$ by the following relations:
\begin{center}
\begin{tabular}{l c l}
$(1)$ & $pp \equiv p$ & for all $p$ \vspace{.1in}\\
$(2)$ & $pqp \equiv qpq$ & for all $p$ and $q$\vspace{.1in} \\
 $(3)$ & $pqs \equiv qps$ and $sqp \equiv spq$ & whenever $p<s<q$.
 %$(3)$ & $pq \equiv qp$ & for $|p-q|>1$.
\end{tabular}
\end{center}

This monoid is better suited for our purposes than Hecke monoid of \cite{BKSTY}, see Remark \ref{rem:whyrestrict}.

%\begin{remark}
% This monoid differs from the one defined in \cite{BKSTY}. In \cite{BKSTY}, the third relation is 
%\begin{center}
%\begin{tabular}{l c l}
%$(3')$ & $pq \equiv qp$ & for $|p-q|>1$.
%\end{tabular}
%\end{center}
%The definition in \cite{B} is made so that for every word in the monoid, there is a unique permutation $b \in S_\infty=\bigcup_{n\geq 1} S_n$ 
%such that the word is equivalent to a reduced word for $b$. For our purposes and for the purposes of \cite{BS}, the above 
%choice of relations is better for reasons to be seen soon,
%see Remark \ref{rem:whyrestrict}.
%Relation $(3)$ is the Knuth relation already mentioned above, which defines the {\it {plactic monoid}}, see \cite{LTh}.
%\end{remark}

We shall say two words are {\it {$K$-Knuth equivalent}} if they are equal in the $K$-Knuth monoid. We denote $K$-Knuth equivalence by $\equiv$. 
We shall also say two words are {\it {insertion equivalent}} if they Hecke insert into the same tableau. We shall denote insertion equivalence by $\sim$.

\begin{example} \label{ex:3124}
 The words $34124$ and $3124$ are $K$-Knuth equivalent, since $$34124 \equiv 31424 \equiv 31242 \equiv 13242 \equiv 13422 \equiv 1342 \equiv 1324 \equiv 3124.$$
They are not insertion equivalent, however, since they insert into the following two distinct tableaux.

\begin{center}
$P(34124)=$
 \begin{ytableau}
  1 & 2 & 4 \\
3 & 4
 \end{ytableau}\hspace{1in}
$P(3124)=$
\begin{ytableau}
  1 & 2 & 4 \\
3 
 \end{ytableau}\hspace{.1in}.
\end{center}
\end{example}

\begin{example} \label{ex:35}
 As we soon shall see, $13524 \not \equiv 15324$.
\end{example}

\subsection{Properties of $K$-Knuth equivalence}

We will need three additional properties of Hecke insertion and $K$-Knuth equivalence.
%The first implies that if two Hecke words insert into the same tableau, then they are equivalent.
The first follows from \cite[Theorem 6.2]{BS}.

\begin{theorem} \label{cor:imply}
 Insertion equivalence implies $K$-Knuth equivalence: if $w_1\sim w_2$ for words $w_1$ and $w_2$, then $w_1\equiv w_2$.
\end{theorem}

As we saw in Example \ref{ex:3124}, the converse of this result is not true. 

We now examine the length of the longest strictly increasing subsequence of a word $w$, denoted by $\i(w)$,
and length of the longest strictly decreasing subsequence of $w$, $\d(w)$. The next result follows from the 
$K$-Knuth equivalence relations.

\begin{lemma}\label{lem:LIS}
If $w_1\equiv w_2$, then $\i(w_1)=\i(w_2)$ and $\d(w_1)=\d(w_2)$.
\end{lemma}

\begin{proof}
It is enough to assume the two words differ by one equivalence relation. 

Suppose $w_1=upv$ and $w_2=uppv$ for some possibly empty words $u$ and $v$. Then if $u'pv'$ is a
strictly increasing sequence in $w_1$, for some possibly empty $u'$ subword of $u$ and $v'$ subword of $v$,
 it is also a strictly increasing sequence in $w_2$. And 
since a strictly increasing sequence can only use one occurence of $p$, any strictly increasing sequence in $w_2$
is also strictly increasing in $w_1$. 

Next, consider the case where $w_1=upqpv$ and $w_2=uqpqv$, and assume $p<q$. If $u'pqv'$ is a strictly increasing sequence
in $w_1$, notice that we have the same sequence in $w_2$. Similarly, if a strictly increasing
sequence of $w_1$ is of the form $u'pv'$ or $u'qv'$, we have the same sequence in $w_2$. Since strictly increasing
sequences of $w_2$ involving the $p$ or $q$ that are outside of $u$ and $v$ have the same form as those described above, 
any strictly increasing sequence of $w_2$ is appears as a strictly increasing sequence in $w_1$.

Lastly, suppose $w_1=upqsv$ and $w_2=uqpsv$ for $p<s<q$. If a strictly increasing sequence
in $w_1$ (resp. $w_2$) uses only one of the $p$ and $q$ outside of $u$ and $v$, then clearly this is still a strictly increasing
sequence in $w_2$ (resp. $w_1$). If a strictly increasing sequence in $w_1$ is $u'pqv'$, then $u'psv'$ is a strictly
increasing sequence in $w_2$ of the same length and vice versa.

A similar arguement applies for $\d(w_1)$ and $\d(w_2)$.
\end{proof}

We can use this result to verify that $13524$ is not $K$-Knuth equivalent to $15324$, as promised in Example \ref{ex:35}. 
Indeed, $\d(13524)=2$ and $\d(15324)=3$.

\begin{remark}
 We do not know any analogue of the other Greene-Kleitman invariants, see \cite{GK} and \cite{ThY}. 
\end{remark}

We shall need the following lemma.

\begin{lemma}\label{restricttointerval}\cite[Lemma 5.5]{BS}
 Let $I$ be an interval in the alphabet $A$. If $w\equiv w'$, then $w|_I \equiv w'|_I$.
\end{lemma}
%\begin{proof}
% It suffices to assume $w'$ differs from $w$ by a single restricted Hecke equivalence relation. 
%The result is obvious for $(1)$. For $(2)$, if $p\in I$ and $q\notin I$, $pqp$ and $qpq$ both become $p$.
%For $(3')$ suppose $p<s<q$, if $p\notin I$ we are left with $qs=qs$ and if $q\notin I$, we are left with $ps=ps$.
%\end{proof}

The last result in this section was proved by H. Thomas and A. Yong in \cite{ThY}. It gives information about 
the shape of $P(w)$ and of $P(h)$ for any $h\equiv w$.

\begin{theorem}\cite[Theorem 1.3]{ThY}\label{thm:heckeshape}
 For any word $w$, the size of the first row and first column of its insertion tableau are given by
 $\i(w)$ and $\d(w)$, respectively.
\end{theorem}

\section{$K$-theoretic Poirier-Reutenauer}

Let $[[h]]$ denote the sum of all words in the $K$-Knuth equivalence class of an initial word $h$:
$$[[h]] = \displaystyle\sum_{h \equiv w} w.$$
This is an infinite sum. The number of terms in $[[h]]$ of length $l$ is finite, however, for every positive integer $l$. 

Let $KPR$ denote the vector space 
spanned by all sums of the form $[[h]]$ for some initial word $h$. We will endow $KPR$ with a product and a coproduct structure, which are compatible 
with each other. We will refer to the resulting bialgebra as the {\it {$K$-theoretic Poirier-Reutenauer bialgebra}} and denote it by $KPR$. 

\subsection{$K$-Knuth equivalence of tableaux}

Suppose we have increasing tableaux $T$ and $T'$. Recall that $\r(T)$ denotes the row reading word of $T$. As in \cite{BS}, we say that $T\equiv T'$ if $\r(T)\equiv \r(T')$.

\begin{example}
For the $T$ and $T'$ shown below, we have that $T\equiv T'$ because $$\r(T)=34124\equiv \r(T')=3124$$ as shown in Example \ref{ex:3124}.
\begin{center}
$T=$
\begin{ytableau}
1 & 2 & 4 \\
3 & 4
\end{ytableau}\hspace{.5in}$\equiv$\hspace{.5in}
$T'=$
\begin{ytableau}
1 & 2 & 4 \\
3
\end{ytableau}
\end{center}
\end{example}

Note that by Lemma \ref{lem:LIS} and Theorem \ref{thm:heckeshape}, if two tableaux are equivalent, their first rows
have the same size and their first columns have the same size.

The following lemma says that each element of $KPR$ splits into insertion classes of words.

\begin{lemma} \label{lem:hecktoins}
 We have 
$$[[h]] = \sum_{T} \left( \sum_{P(w)=T} w\right)$$
where the sum is over all increasing tableaux $T$ whose reading word is in the $K$-Knuth equivalence class of $h$. 
\end{lemma}

\begin{proof}
 Follows from Theorem \ref{cor:imply}.
\end{proof}

This expansion is always finite by Lemma \ref{lem:fin}.

\subsection{Product structure}

Let $\shuffle$ denote the usual shuffle product of words. Let $h$ be a word in the alphabet $[n]$, and let $h'$ be a word in the alphabet $[m]$. 
Denote by $w[n]$ the word obtained from $w$ by increasing each letter by $n$. Define 
$$[[h]] \cdot [[h']] = \sum_{w \equiv h, w' \equiv h'} w \shuffle w'[n].$$

\begin{theorem} \label{thm:PRprod}
For any two initial words $h$ and $h'$, their product can be written as 
$$[[h]] \cdot [[h']] = \displaystyle\sum_{h''} [[h'']],$$
where the sum is over a certain set of initial words $h''$. 
\end{theorem}

\begin{proof}
%We have 
%$$\left( \displaystyle\sum_{h \equiv w} w\right) \cdot \left(\displaystyle\sum_{h'\equiv w} w\right)=
%\left( \displaystyle\sum_{h \equiv w} w\right) \shuffle \left(\displaystyle\sum_{h'[n]\equiv w} w\right).$$
From Lemma \ref{restricttointerval}, 
we know that if a word appears in the righthand sum, the entire equivalence class of this word appears as well.
The claim follows. 
\end{proof}

\begin{example}\label{ex:PRprod}
 Let $h=12$, $h'=312$. Then 
$$[[12]] \cdot [[312]] = [[53124]]+[[51234]]+[[35124]]+[[351234]]+[[53412]]+[[5351234]].$$  
%$$[[12]] \cdot [[312]] = 12534+12354+15234+ \ldots = [[34124]]+[[351234]]+[[535124]]+[[5351234]].$$
\end{example}

\begin{theorem} \label{thm:hecktoins}
 Let $h$ be a word in alphabet $[n]$, and let $h'$ be a word in alphabet $[m]$. Suppose $\mathscr{T}=\{P(h),T'_1, T'_2, \ldots, T'_s\}$ is 
the equivalence class containing $P(h)$. Then we have 
$$[[h]] \cdot [[h']] =
\displaystyle\sum_{T\in T(h \shuffle h')}\sum_{P(w)=T} w,$$
where $T(h \shuffle h')$ is the finite set of tableaux $T$ such that $T|_{[n]}\in\mathscr{T}$ and $\r(T)|_{[n+1,n+m]}\equiv h'[n]$.
\end{theorem}

\begin{proof}
 If $w$ is a shuffle of some $w_1\equiv h$ and $w_2\equiv h'[n]$, then by Lemma \ref{lem:restrict} $P(w)|_{[n]}=P(w_1)\in\mathscr{T}$.
By Lemma \ref{lem:hecktoins} and Theorem \ref{thm:PRprod}, we get the desired expansion. 
Its finiteness follows from Lemma \ref{lem:fin}.
\end{proof}

\begin{example}
 Let's take $h=12$ and $h'=312$. Then \\
\begin{center}
 $P(h)=$ \begin{ytableau}
          1 & 2
         \end{ytableau}\hspace{.1in}and $P(h')=$
\begin{ytableau}
1 & 2 \\
3
\end{ytableau}\hspace{.1in}.
\end{center}
The insertion tableaux appearing in their product are those shown below.
\begin{center}
\begin{ytableau}
 1 & 2 & 4 \\
3 \\
5
\end{ytableau}\hspace{.2in}
\begin{ytableau}
  1 & 2 & 4 \\
3 & 5 \\
5
 \end{ytableau}\hspace{.2in}
\begin{ytableau}
 1 & 2 & 4 \\
3 & 4 \\
5
\end{ytableau}\hspace{.2in}
\begin{ytableau}
 1 & 2 & 3 & 4 \\
5
\end{ytableau}\hspace{.2in}
 \begin{ytableau}
  1 & 2 & 4 \\
3 & 5
 \end{ytableau}\hspace{.2in}
\begin{ytableau}
  1 & 2 & 3 & 4 \\
3 & 5
 \end{ytableau}\hspace{.2in}
\end{center}
\begin{center}
\begin{ytableau}
 1 & 2 \\
3 & 4 \\
5
\end{ytableau}\hspace{.2in}
\begin{ytableau}
 1 & 2 & 3 & 4 \\
3 \\
5
\end{ytableau}\hspace{.2in}
\begin{ytableau}
  1 & 2 & 3 & 4 \\
3 & 5 \\
5
 \end{ytableau}
\end{center}

Each of them restricted to $[2]$ is clearly $P(12)$. One can check that each of the row reading words restricted to the alphabet $3,4,5$ is $K$-Knuth 
equivalent to $534$. For example, in case of the last tableau
$$53534 \equiv 53354 \equiv 5354 \equiv 5534 \equiv 534.$$ 

Note that the first three tableaux listed are equivalent to each other and the last two tableaux listed are equivalent to each other. 
We will see in the next section that the fourth and sixth tableaux are not 
equivalent. With this in mind, we can see that there are no other equivalent pairs by examining the sizes of the first rows and first columns. 
The six classes of tableaux in this example correspond to the six equivalence classes in Example \ref{ex:PRprod}.
\end{example}

\begin{corollary} \label{cor:prod}
 The vector space $KPR$ is closed under the product operation. That is, the sum appearing on the right hand side in Theorem \ref{thm:PRprod}  is always finite. 
\end{corollary}

\begin{proof}
 We know from Lemma \ref{lem:hecktoins} that $K$-Knuth classes are coarser than insertion classes. Thus finiteness of right hand side in Theorem \ref{thm:PRprod}
follows from that in Theorem \ref{thm:hecktoins}. 
%The latter is true by Lemma \ref{lem:fin}.
\end{proof}

\begin{remark} \label{rem:prod}
 The product of insertion classes is not necessarily a linear combination of insertion classes. For example, consider the following tableaux, $T$ and $T'$.
\begin{center}
$T=$
 \begin{ytableau}
  1 & 2
 \end{ytableau}\hspace{1in}
$T'=$
\begin{ytableau}
 1 & 2 & 4 \\
3
\end{ytableau}
\end{center}
Then $12$ and $1342$ are in the insertion classes of $T$ and $T'$, respectively, and we get $315642$ as a term in their shuffle product.
The insertion tableau of $315642$ is shown below.
\begin{center}
$P(315642)=$
 \begin{ytableau}
  1 & 2 & 6 \\
3 & 4 \\
5
 \end{ytableau}
\end{center}
Notice that $P(3156442)=P(315642)$, but $3156442$ will not appear in the shuffle product of the insertion classes of $12$ and $1342$ since
$314562|_{[3,6]}=35644\not\sim3564$.
\end{remark}

\subsection{Coproduct structure}

For any word $w$, let $\overline w$ denote the {\it {standardization}} of $w$: if a letter $a$ is $k$-th smallest among the letters of 
$w$, it becomes $k$ in $\overline w$. For example, ${\overline {42254}} = 21132$. Note that standardization of a word is 
always an initial word.

Let $w = a_1 a_2 \ldots a_n$ be an initial word. Define $$\Delta (w) = \sum_{i=0}^n \overline {a_1 \ldots a_i} \otimes \overline {a_{i+1} \ldots a_n}.$$

Similarly, define
$$\Delta ([[h]]) = \sum_{h \equiv w} \Delta (w).$$

\begin{example}
 We have $$\Delta(34124) = \emptyset \otimes 34124 + 1 \otimes 3123 + 12 \otimes 123 + 231 \otimes 12 + 3412 \otimes 1 + 34124 \otimes \emptyset,$$
and 
$$\Delta([[34124]]) = \Delta(34124) + \Delta(31424) + \Delta(31242) + \dotsc.$$
\end{example}

Here, $\emptyset$ should be understood to be the identity element of the ground field. We denote it by $\emptyset$ so as to avoid confusion with the word $1$.

\begin{theorem} \label{thm:PRcoprod}
For any initial word $h$, its coproduct can be written as 
$$\Delta ([[h]]) = \sum_{h',h''} [[h']] \otimes[[h'']],$$
where the sum is over a certain set of pairs of initial words $h',h''$. 
\end{theorem}

\begin{proof}
 For $w = a_1 a_2 \ldots a_n$ let $\blacktriangle(w) = \sum_{i=0}^n  {a_1 \ldots a_i} \otimes  {a_{i+1} \ldots a_n}$, and $\blacktriangle ([[h]]) = \sum_{h \equiv w} \blacktriangle (w).$
It is clear that $\blacktriangle ([[h]]) = \sum_{h',h''} [[h']] \otimes[[h'']]$ for some collection of pairs of words $h',h''$. This is because $K$-Knuth 
equivalence relations are local and thus can be applied on both sides of $\otimes$ in parallel with applying the same relation to the corresponding word on the left. 
It remains to standardize every term on the right and to use the fact that $K$-Knuth equivalence relations commute with standardization.
\end{proof}

\begin{example}
If we take $h=12$ in the previous theorem, we have
$$\Delta([[12]])=[[\emptyset]]\otimes[[12]]+[[1]]\otimes[[1]]+[[12]]\otimes[[1]]+[[1]]\otimes[[12]]+[[12]]\otimes[[\emptyset]].$$
%If we take $h=312$ in the previous theorem, we have
%\begin{eqnarray}
%\Delta ([[312]])&=&[[\emptyset]]\otimes [[312]] + [[1]]\otimes [[12]] + [[1]]\otimes [[21]]+[[1]]\otimes[[132]]+[[12]]\otimes [[1]]+ [[12]]\otimes [[12]] \nonumber \\
%&+& [[21]]\otimes[[1]]+[[21]]\otimes[[21]]+[[212]]\otimes [[1]]+[[312]]\otimes[[1]]+[[312]]\otimes [[\emptyset]].\nonumber
%\end{eqnarray}
%NICE EXAMPLE, POSSIBLY $34124$ FROM PREVIOUS EXAMPLE, POSSIBLY SOME OTHER (Are all of the terms here?)
\end{example}

\begin{theorem} \label{thm:cohecktoins}
 Let $h$ be a word. We have 
$$\Delta ([[h]]) =
\displaystyle\sum_{(T',T'') \in T(h)} \left( \sum_{P(w)= \overline{T'}} w \right) \otimes \left( \sum_{P(w)= \overline{T''}} w \right),$$
where $T(h)$ is the finite set of pairs of tableaux $T',T''$ such that $\r(T') \r(T'') \equiv h$.
\end{theorem}

\begin{proof}
As we have seen in the proof of Theorem \ref{thm:PRcoprod}, $\blacktriangle ([[h]]) = \sum_{h',h''} [[h']] \otimes[[h'']]$. Note that the sum on the right
is multiplicity-free and that if we split each of the $[[h']]$ and $[[h'']]$ into insertion classes, we get exactly 
$$\displaystyle\sum_{(T',T'') \in T(h)} \left( \sum_{P(w)= {T'}} w \right) \otimes \left( \sum_{P(w)= {T''}} w \right).$$
 It remains to apply standardization to get the desired result. 
The finitness follows from Lemma \ref{lem:fin}.
\end{proof} 

\begin{corollary} \label{cor:coprod}
The vector space $KPR$ is closed under the coproduct operation. That is, the sums appearing on the right hand side in Theorem \ref{thm:PRcoprod} are always finite. 
\end{corollary}
\begin{proof}
Entries in the tableaux $T$ and $T'$ are a subset of letters in the word $h$. The statement follows from finitness in Theorem \ref{thm:cohecktoins} and the fact that 
$K$-Knuth classes are coarser than insertion classes.
\end{proof}

\begin{remark} \label{rem:coprod}
It is not true that insertion classes are closed under the coproduct. For example, $123\otimes 1 $ is a term in 
$\Delta(1342)$ and thus in the coproduct of its insertion class, but $123 \otimes 11$ is not. To see this, consider 
all words $h$ containing only $1$, $2$, $3$, and $4$ such that $123\otimes 11$ is in $\Delta(h)$. These words are 
$12344$, $12433$, $13422$, and $23411$, none of which are insertion equivalent to $1342$.
\end{remark}

\subsection{Compatibility and antipode}

Recall that a product $\cdot$ and a coproduct $\Delta$ are {\it {compatible}} if the coproduct is an algebra morphism:
$$\Delta(X \cdot Y) = \Delta(X) \cdot \Delta(Y).$$
A vector space endowed with a compatible product and coproduct is called a {\it {bialgebra}}.
We refer the reader for example to \cite{ShS} for details on bialgebras. 

\begin{theorem}
 The product and coproduct structures on $KPR$ defined above are compatible, thus giving $KPR$ a bialgebra structure.
\end{theorem}

\begin{proof}
 The result follows from the fact that the same is true for initial words. Indeed, if $u = a_1 \ldots a_n$ and $w=b_1 \ldots b_m$ are two initial words, 
then 
$$\Delta(w \cdot u) = \Delta(w \shuffle u[n]) = \sum_{v = c_1 \ldots c_{n+m}} \sum_{i=1}^{n+m} \overline {c_1 \ldots c_i} \otimes \overline {c_{i+1} \ldots c_{n+m}},$$
where $v$ ranges over shuffles of $w$ and $u[n]$. On the other hand,  
$$\Delta(w) \cdot \Delta(u) = \left(\sum_{i=1^n} \overline  {a_1 \ldots a_i} \otimes \overline {a_{i+1} \ldots a_n}\right) \cdot 
\left(\sum_{j=1^m} \overline {b_1 \ldots b_j} \otimes \overline {b_{j+1} \ldots b_m}\right)=$$
$$= \sum_{i,j} (\overline {a_1 \ldots a_i} \shuffle \overline {b_1 \ldots b_j} [i]) \otimes (\overline {a_{i+1} \ldots a_n} \shuffle \overline {b_{j+1} \ldots b_m}[n-i]).$$
The two expressions are easily seen to be equal. 
\end{proof}

We also remark that $KPR$ has no antipode (see \cite{ShS} for a definition). Indeed, assume $S$ is an antipode. Then since 
$$\Delta([[1]]) = \emptyset \otimes [[1]] + [[1]] \otimes [[1]] + [[1]] \otimes \emptyset,$$
we solve 
$$S([[1]]) = - \frac{[[1]]}{\emptyset + [[1]]}.$$

This final expression is not a finite linear combination of basis elements of $KPR$, and thus does not lie in the bialgebra.  

\section{Unique Rectification Targets}

As we have seen, $K$-Knuth equivalence classes may have several corresponding insertion tableaux. The following is an open problem.

\begin{problem}
 Describe $K$-Knuth equivalence classes of increasing tableaux. 
\end{problem}

Of special importance are the $K$-Knuth equivalence classes with only one element. 
%As we will see, those allow us to formulate 
%a version of Littlewood-Richardson rule for $K$-theory of Grassmannians. 

\subsection{Definition and examples}

We call $T$ a \textit{unique rectification target} or a URT if it is the only tableau in its $K$-Knuth equivalence class \cite[Definition 3.5]{BS}. In other words,
$T$ is a URT if for every $w\equiv \r(T)$ we have $P(w)=T$. The terminology is natural in the context of the $K$-theoretic jeu de taquin of Thomas and Yong \cite{ThY}.
If $P(w)$ is a URT, we call the equivalence class of $w$ a unique rectification class.

For example, $P(1342)$ is not a unique rectification target
because $3124\equiv 34124$, as shown in Example \ref{ex:3124}, and $P(3124)\neq P(34124)$ as shown below. 
\begin{center}
$P(3124)=$
\begin{ytableau}
 1 & 2 & 4 \\
3
\end{ytableau}\hspace{1in}
$P(34124)=$
\begin{ytableau}
1 & 2 & 4 \\
3 & 4
\end{ytableau}
\end{center}

It follows that $[[3124]]$ is not a unique rectification class.

In \cite{BS}, Buch and Samuels give a uniform construction of unique rectification targets of any shape as follows. Define the \textit{minimal increasing tableau} $M_\lambda$
of shape $\lambda$ by filling the boxes of $\lambda$ with the smallest possible values allowed in an increasing tableau. In other words, $M_\lambda$ is the tableau obtained by 
filling all boxes in the $k$th southwest to northeast diagonal of $\lambda$ with 
positive integer $k$.

\begin{example}
 The tableaux below are minimal increasing tableaux.
\begin{center}
$M_{(3,2,1)}=$
\begin{ytableau}
 1 & 2 & 3 \\
2 & 3 \\
3
\end{ytableau}\hspace{1in}
$M_{(5,2,1,1)}=$
\begin{ytableau}
1 & 2 & 3 & 4 & 5 \\
2 & 3 \\
3 \\
4
\end{ytableau}
\end{center}
\end{example}

In \cite[Theorem 1.2]{ThY2}, Thomas and Yong prove that the superstandard tableaux of shape $\lambda$, $S_\lambda$, is a URT, where $S_\lambda$ is defined to be 
the standard Young tableau with $1,2,\ldots,\lambda_1$ in the first row, $\lambda_1+1,\lambda_1+2,\ldots, \lambda_1+\lambda_2$ in the second row, etc. 

\begin{example} The following are superstandard tableaux.
 \begin{center}
$S_{(3,2,1)}=$
\begin{ytableau}
 1 & 2 & 3 \\
4 & 5 \\
6
\end{ytableau}\hspace{1in}
$S_{(5,2,1,1)}=$
\begin{ytableau}
 1 & 2 & 3 & 4 & 5 \\
6 & 7 \\
8 \\
9
\end{ytableau}
\end{center}
\end{example}

\begin{problem}
 Characterize all unique rectification targets or at least provide an efficient algorithm to determine if a given tableau is a URT.
\end{problem}

\begin{remark} \label{rem:whyrestrict}
 Note that if one uses the less restrictive Hecke equivalence of \cite{BKSTY} instead of the $K$-knuth equivalence of \cite{BS}, URT are extremely scarce.
For example, the tableau with reading word $3412$ is equivalent to the tableau with reading word $3124$. In fact, with this definition
there is no standard URT of shape $(2,2)$.  
\end{remark}

\subsection{Product and coproduct of unique rectification classes}

As we have seen before, the product and coproduct of insertion classes do not necessarily decompose into insertion classes. However, the story is different 
if the classes are unique rectification classes, as seen in the following theorems.

\begin{theorem} \label{thm:ins}
% Let $h$ be a word in alphabet $[n]$, and let $h'$ be a word in alphabet $[m]$. Suppose $\mathscr{T}=\{P(h),T'_1, T'_2, \ldots, T'_s\}$ is 
%the equivalence class the insertion tableau of $h$, $P(h)$. 
Let $T_1$ and $T_2$ be two URT. 
Then we have 
$$\left(\sum_{P(w)=T_1} w \right) \cdot \left(\sum_{P(w)=T_2} w \right) =
\displaystyle\sum_{T\in T(T_1 \shuffle T_2)}\sum_{P(w)=T} w,$$
where $T(T_1 \shuffle T_2)$ is the finite set of tableaux $T$ such that $T|_{[n]} = T_1$ and $P({\r(T)|_{[n+1,n+m]}}) = T_2$.
\end{theorem}

\begin{proof}
 Since $T_1$ and $T_2$ are URTs, the left hand side is $([[\r(T_1)]])\cdot([[\r(T_2)]])$ and  $T(T_1 \shuffle T_2)$ is $T(\r(T_1)\shuffle\r(T_2))$ as
in Theorem \ref{thm:hecktoins}.
\end{proof}

\begin{theorem} \label{thm:coins}
% Let $h$ be a word. We have 
Let $T_0$ be a URT. We have 
$$\Delta \left(\sum_{P(w)=T_0} w \right) =
\displaystyle\sum_{(T',T'') \in T(T_0)} \left( \sum_{P(w)= \overline{T'}} w \right) \otimes \left( \sum_{P(w)= \overline{T''}} w \right),$$
where $T(T_0)$ is the finite set of pairs of tableaux $T',T''$ such that $P(\r(T') \r(T'')) = T_0$.
\end{theorem}

\begin{proof}
 Since $T_0$ is a URT, the left hand term is $\Delta([[\r(T_0)]])$ and $T(T_0)=T(\r(T_0))$ as described in Theorem \ref{thm:cohecktoins}.
\end{proof}

%If $P(w')$ and $P(w'')$ are solitary, we can state the result in terms of insertion classes 
%instead of restricted Hecke equivalence classes, thus giving a product formula for solitary tableaux.

%\begin{corollary}
%If $T=P(h)$ and $T'=P(h')$ are solitary tableaux, then we have 
%$$\left( \displaystyle\sum_{P(w)=T} w\right) \cdot \left(\displaystyle\sum_{P(w)=T'} w\right)=
%\displaystyle\sum_{T\in h \shuffle h'}\sum_{P(w)=T} w,$$
%where $h \shuffle h'$ is the set of tableaux $T''$ such that $T''|_{[n]}\in\mathscr{T}$ and $T''|_{[n+1,n+m]}\equiv h'[n]$.
%\end{corollary}

\begin{remark}
Note that a product of unique rectification classes is not necessarily  
a sum of unique rectification classes. For example, if we let $w'=12$ and $w''[2]=34$, then $13422$ appears
in the shuffle product. One checks $P(13422)$ is one of the tableaux in Example \ref{ex:3124}  
and thus is not a URT.

Similarly, the coproduct of a unique rectification class does not necessarily decompose into
unique rectification classes. Consider $T_0$, $T'$, $T''$, and $T'''$ below.
\begin{center}
 $T_0=$
\begin{ytableau}
  1 & 2 & 4 \\
3 & 5
 \end{ytableau}\hspace{.5in}
$T'=$
\begin{ytableau}
 1 & 2 & 5 \\
3
\end{ytableau}\hspace{.5in}
$T''=$
\begin{ytableau}
2 & 4 
\end{ytableau}\hspace{.5in}
$T'''=$
\begin{ytableau}
1 & 2 & 5 \\
3 & 5
\end{ytableau}

\end{center}
One can check that $T_0$ is a URT and $P(\r(T')\r(T''))=P(312524)=T_0$, but 
$T'$ is not a URT since it is equivalent to $T'''$.
\end{remark}

%\subsection{Three questions pertaining to solitary tableaux}

%We have already remarked on the similarity of our theorems and those of \cite{ThY2,ThY3}, where the $K$-theoretic version of jeu de taquin is used. We state the following
%three questions which may help further illuminate the relation.

%\begin{question}
% Is it true that if jeu de taquin rectification produces a solitary tableau, then any other order of rectification would also produce the same tableau? 
%An analogous statement for the special case of superstandard tableaux is \cite[Theorem 1.2]{ThY2}. 
%\end{question}

%\begin{question}
% It is known from \cite{ThY2} that the result of jeu de taquin rectification depends on the chosen order. Orders correspond to increasing tableaux in a natural way. Is 
%it true that if we restrict to {{solitary}} tableaux orders, the result does not depend on the order anymore?
%\end{question}

%\begin{question}
% Is it true that all increasing tableaux of a rectangular shape are solitary? This is interesting in light of the previous question and \cite[Theorem 1.2]{ThY3}.
%\end{question}

\section{Connection to symmetric functions}

\subsection{Symmetric functions and stable Grothendieck polynomials}

\

We denote the ring of symmetric functions in an infinite number of variables $x_i$ by $$\Lambda = \Lambda(x_1, x_2, \ldots) = \oplus_{n \geq 0} \Lambda_n.$$ 
The $n^{\text{th}}$ graded component, $\Lambda_n$, consists of homogenous elements of degree $n$ with $\Lambda_0 = \mathbb R$.
There are several important bases of $\Lambda$ indexed by partitions $\lambda$ of integers. The two most notable such bases
are the monomial symmetric functions, $m_{\lambda}$, and Schur functions, $s_{\lambda}$.
We refer the reader to \cite{EC2} for definitions and further details on the ring $\Lambda$. 

For $f(x_1, x_2, \ldots) \in \Lambda$ let 
$$\Delta(f) = f(y_1,\ldots; z_1, \ldots) \in \Lambda(y_1, \ldots) \otimes \Lambda(z_1, \ldots)$$ be the result of splitting the alphabet of $x_i$'s
into two disjoint alphabets of $y_i$'s and $z_i$'s. $\Lambda$ is known to be a bialgebra with this coproduct, see \cite{Zel}.

Let us denote by $\hat \Lambda$ the completion of $\Lambda$, which consists of possibly infinite linear combinations of $m_{\lambda}$'s. 
Each element of $\hat \Lambda$ can be split into graded components, each being a finite linear combination of the $m_{\lambda}$'s. 
Also, let $\hat {\Lambda \otimes \Lambda}$ denote the completion of $\Lambda \otimes \Lambda$, 
consisting of possibly infinite linear combinations of $m_{\lambda} \otimes m_{\mu}$'s.
It is not hard to see that the completion $\hat \Lambda$ inherits a bialgebra structure from $\Lambda$ in the following sense. 

\begin{theorem}
 If $f,g  \in \hat \Lambda$ then $$f \cdot g  = \sum_{\mu} c_{\mu} m_{\mu} \in \hat \Lambda.$$ If $f \in \hat \Lambda$, then 
$$\Delta(f) = \sum_{\mu, \nu} c_{\mu, \nu} m_{\mu} \otimes m_{\nu} \in \hat {\Lambda \otimes \Lambda}.$$ Furthermore, the 
coefficients $c_{\mu}$ and $c_{\mu, \nu}$ in both expressions are unique. 
\end{theorem}

\begin{proof}
 It is easy to see that each $m_{\mu}$ in the first case and each $m_{\mu} \otimes m_{\nu}$ in the second case can appear only in finitely many terms on the left. 
The claim follows. 
\end{proof}

Recall the definition of a set-valued tableau given in Section \ref{sec:rec}. Given a set-valued tableau $T$, let 
$x^T$ be the monomial in which the exponent of $x_i$ is the number of occurences of the letter $i$ in $T$. Let $|T|$ be the 
degree of this monomial.
\begin{example}
The tableau shown below has $x^T=x_1x_3x_4x_5^2x_6^3x_8x_9$ and $|T|=11$.
\begin{center}
 \begin{ytableau}
  134 & 4 & 568 \\
56 & 6 \\
9
 \end{ytableau}
\end{center}
\end{example}

In \cite{B}, Buch proves a combinatorial interpretation of the single stable Grothendieck polynomials indexed by partitions, $G_\lambda$, which we 
present as the definition. This interpretation is implicitly present in the earlier paper \cite{FG}.

\begin{theorem}\cite[Theorem 3.1]{B}
 The single stable Grothendieck polynomial $G_{\lambda}$ is given by the formula 
$$G_\lambda=\displaystyle\sum_T(-1)^{|T|-|\lambda|}x^T,$$
where the sum is over all set-valued tableaux $T$ of shape $\lambda$.
\end{theorem}

\begin{example}
 We have $$G_{(2,1)}=x_1^2x_2+2x_1x_2x_3-x_1^2x_2^2-2x_1^2x_2x_3-8x_1x_2x_3x_4+\ldots,$$ where, for example 
the coefficient of $x_1^2x_2x_3$ is $-2$ because of the tableaux shown below and the fact that for each of them, $|T|-|\lambda|=1$.
\begin{center}
\begin{ytableau}
 1 & 1 2\\
3
\end{ytableau}\hspace{1in}
\begin{ytableau}
 1 & 13 \\
2
\end{ytableau}
\end{center}
\end{example}

The following claim is not surprising. 

\begin{lemma}
 Each element $f \in \hat \Lambda$ can uniquely be written as 
$$f = \sum_{\mu} c_{\mu} G_{\mu}.$$
Similarly, each element of $\hat {\Lambda \otimes \Lambda}$ can uniquely be written as $$\sum_{\mu, \nu} c_{\mu, \nu} G_{\mu} \otimes G_{\nu}.$$
\end{lemma}

\begin{proof}
Fix any complete order on monomials $m_{\lambda}$ that agrees with the reverse dominance order for a fixed size $|\lambda|$ 
and satisfies $m_{\mu} < m_{\lambda}$ for $|\mu|<|\lambda|$. See, for example, \cite{MS} for details. 
Then $m_{\lambda}$ is the minimal term of $G_{\lambda}$, and we can uniquely recover coefficients $c_{\mu}$ by using $G_{\lambda}$'s
to eliminate minimal terms in $f$. The proof of the second claim is similar.
\end{proof}

What is surprising, however, is the following two theorems proven by Buch in \cite{B}. A priori, the products and the coproducts of $G_{\lambda}$'s
do not have to decompose into {\emph {finite}} linear combinations.  

\begin{theorem} \cite[Corollary 5.5]{B} \label{thm:prodfin}
 We have $$G_{\lambda} G_{\mu} = \sum_{\nu} c_{\lambda,\mu}^{\nu} G_{\nu},$$ where the sum on the right is over a finite set of partition shapes $\nu$.
\end{theorem}

\begin{theorem} \cite[Corollary 6.7]{B} \label{thm:coprodfin}
 We have $$\Delta(G_{\nu}) = \sum_{\nu} d_{\lambda,\mu}^{\nu} G_{\lambda} \otimes G_{\mu},$$
where the sum on the right is over a finite set of pairs $\lambda, \mu$.
\end{theorem}

\subsection{Weak set-valued tableax}
A \textit{weak set-valued tableau} $T$ of shape $\lambda$ is a filling of the boxes
with finite, non-empty multisets of positive integers so that
\begin{enumerate}
 \item the smallest number in each box is greater than or equal to the largest
number in the box directly to the left of it (if that box is present), and
\item the smallest number in each box is strictly greater than the largest number 
in the box directly above it (if that box is present).
\end{enumerate}

Note that the numbers in each box are not necessarily distinct.

For a weak set-valued tableau $T$, define $x^T$ to be $\prod_{i\geq 1} x_i^{a_i}$, where $a_i$ is
the number of occurences of the letter $i$ in $T$.

\begin{example} The following weak set-valued tableau $T$ has $x^T=x_1^3x_2^4x_3^2x_4^2x_5x_6x_8$.
\begin{center}
 \begin{ytableau}
  11 & 12 & 2 & 346 \\
223 & 45 & 8
 \end{ytableau}
\end{center}
\end{example}

Let $J_\lambda=\sum_T x^T$ denote the weight generating function of all weak set-valued tableux $T$ of
shape $\lambda$. We will call $J_\lambda$ the \textit{weak stable Grothendieck polynomial} indexed by $\lambda$.

\begin{example}
 We have that $$J_{(2,1)}=x_1^2x_2 + 2x_1x_2x_3 + 2x_1^3x_2 + 3x_1^2x_2^2+2x_1x_2^3+8x_1x_2x_3x_4+\ldots,$$
where, for example, the coefficient of $x_1^2x_2^2$ is $3$ because of the following weak set-valued tableaux.
\begin{center}
 \begin{ytableau}
  11 & 2 \\
2
 \end{ytableau}\hspace{.5in}
 \begin{ytableau}
  1 & 12 \\
2
 \end{ytableau}\hspace{.5in}
 \begin{ytableau}
  1 & 1 \\
22
 \end{ytableau}
\end{center}

\end{example}

\begin{remark}
 In \cite{LP}, weak stable Grothendieck polynomials $J_{\lambda}$ were introduced when studying the effect of standard ring automorphism 
$\omega$ on the stable Grothendieck polynomials $G_{\lambda}$. In particular, it was shown in \cite[Theorem 9.21]{LP} that $J_{\lambda}$ are symmetric functions. Note that 
our current convention for labeling $J_{\lambda}$ differs from that in \cite{LP} by shape transposition.
\end{remark}

\begin{theorem} 
 We have $$J_{\lambda}(x_1, x_2, \ldots) = (-1)^{|\lambda|}G_{\lambda} \left(\frac{-x_1}{1-x_1}, \frac{-x_2}{1-x_2}, \ldots \right).$$
\end{theorem}

\begin{proof}
 There is a correspondence between set-valued tableaux and weak set-valued tableaux as follows. For each set-valued tableau $T$, we can obtain a family of weak set-valued tableaux 
of the same shape, call the family $\mathcal{T}$,
by saying that $W\in\mathcal{T}$ if and only if $W$ can be constructed from $T$ by turning subsets in boxes of $T$ into multisets. 
Conversely, given any weak set-valued tableau, we can find the set-valued tableau it corresponds to by transforming its multisets into subsets containing the same positive integers. 
For example, if we have the $T$ shown below, then $W_1$ and $W_2$ are in $\mathcal{T}$.
\begin{center}
$T=$
\begin{ytableau}
 13 & 4 & 57 \\
4 & 6 
\end{ytableau}\hspace{.5in}
$W_1=$
\begin{ytableau}
 133 & 4 & 57 \\
4 & 666 
\end{ytableau}\hspace{.5in}
$W_2=$
\begin{ytableau}
 13 & 4 & 557 \\
44 & 66 
\end{ytableau}
\end{center}
Thus if $x^T=x_1^{a_1}x_2^{a_2}x_3^{a_3}\ldots$, we have 
$$\displaystyle\sum_{W\in \mathcal{T}} x^W= \left(\frac{x_1}{1-x_1}\right)^{a_1}\left(\frac{x_2}{1-x_2}\right)^{a_2}\left(\frac{x_3}{1-x_3}\right)^{a_3}\ldots$$
since we can choose to repeat any of the $a_i$ $i$'s any positive number of times.

Therefore
$$(-1)^{|\lambda|}G_\lambda\left(\frac{-x_1}{1-x_1}, \frac{-x_2}{1-x_2}, \ldots \right)=(-1)^{|\lambda|}\displaystyle\sum_T  (-1)^{|T|-|\lambda|}\left(\frac{-x}{1-x}\right)^T = 
\displaystyle\sum_T \left(\frac{x}{1-x}\right)^T=J_\lambda(x_1,x_2,\ldots),$$
where the sum is over set-valued tableaux $T$

\end{proof}

\begin{corollary}\label{cor:coincide}
 The structure constants of the rings with bases $G_{\lambda}$ and $J_{\lambda}$ coincide up to sign. In other words, 
$$J_{\lambda} J_{\mu}  = \sum_{\nu} c_{\lambda, \mu}^{\nu} J_{\nu}$$ if and only if
$$G_{\lambda} G_{\mu}  = \sum_{\nu} (-1)^{|\nu|-|\lambda|-|\mu|} c_{\lambda, \mu}^{\nu} G_{\nu}.$$ 
\end{corollary}

\begin{proof}
 In one direction it is clear, in the other follows from 
$$G_{\lambda}(x_1, x_2, \ldots) = (-1)^{|\nu|-|\lambda|-|\mu|}J_{\lambda} \left(\frac{-x_1}{1-x_1}, \frac{-x_2}{1-x_2}, \ldots \right).$$
\end{proof}

 \subsection{Fundamental quasisymmetric functions}
A composition of $n$ is an ordered arrangement of positive integers which sum to $n$. For example, $(3)$, $(1,2)$, $(2,1)$, and $(1,1,1)$ are all of the compositions of $3$.
To a composition $\alpha$ of $n$, we associate $S_\alpha \subset [n-1]$ by letting $$S_\alpha=\{\alpha_1, \alpha_1+\alpha_2, \ldots , \alpha_1+\alpha_2+\ldots +\alpha_{k-1}\}.$$
Conversely, if $S=\{s_1,s_2,\ldots ,s_k\}$ is a subset of $[n-1]$, we associate a composition, $\mathcal{C}(S)$, to $S$ by $$\mathcal{C}(S)=\{s_1, s_2-s_1,s_3-s_2,\ldots ,n-s_k\}.$$
For example, if $S=(1,4,5)\subset[6-1]$, $\mathcal{C}(S)=(1, 4-1, 5-4, 6-5)=(1,3,1,1)$. Conversely, given composition $\alpha=(1,3,1,1)$, $S_\alpha=(1, 1+3,1+3+1)=(1,4,5)$. 

We define the \textit{descent set} of a word $h=h_1h_2\ldots h_l$ to be the set $\mathcal{D}(H)=\{i|h_i>h_{i+1}\}$. Then, using the definitions above, given that $\mathcal{D}(h)=\{\alpha_1,\alpha_2,\ldots,\alpha_m\}$,
we have the associated composition $$\mathcal{C}(h)=\mathcal{C}(\mathcal{D}(h))=(\alpha_1,\alpha_2-\alpha_1, \alpha_3-\alpha_2,\ldots, \alpha_m-\alpha_{m-1}, l-\alpha_m).$$ We call $\mathcal{C}(h)$ the \textit{descent composition}
for $h$.

\begin{example}
 If $h=11423532$, the descent set of $h$ is $\{3,6,7\}$ and $\mathcal{C}(h)=(3,3,1,1)$.
\end{example} 

We shall now define the fundamental quasisymmetric function, $L_\alpha$. Given $\alpha$, a composition of $n$, define 

$$L_\alpha=\displaystyle\sum _{\substack{i_1\leq\ldots\leq i_n \\ i_j<i_{j+1}\text{ if }j\in S_\alpha}}x_{i_1}x_{i_2}\cdots x_{i_n}.$$

For more information on the ring of quasisymmetric functions and on fundamental quasisymmetric functions see \cite{EC2}.

\begin{example}
The fundamental quasisymmetric function indexed by the composition $(1,3)$ is 
$$L_{(1,3)}=x_1x_2^3+x_1x_3^3+x_2x_3^3+x_1x_2^2x_3+x_1x_2x_3^2+x_1x_2x_3x_4+\ldots,$$
an infinite sum where all terms have degree $4$. Note that every term must have $i_1<i_2$ since $S_{(1,3)}=\{2\}$, so $x_1^2x_2^2$ will never appear in $L_{(1,3)}$.
\end{example}

%For more information on $\QSym$ and the basis fundamental quasisymmetric functions, see ??.

Given a weak set-valued tableau $T$ filled with elements of $[n]$, each appearing once, we say that there is a descent at entry $i$ if $i+1$ is strictly below $i$.
We may then find the descent set of $T$ and determine the composition corresponding to its descent set, $\mathcal{C}(T)$, by listing the entries in increasing order and marking
the entries at which there was a descent in the tableau.

\begin{example}\label{ex:tabdescent}
 The descent set of $T$ shown below is $\{3,5,9\}$ and the corresponding composition is $\mathcal{C}(T)=(3,2,4,2)$.
 \begin{center}
\begin{ytableau}
  123 & 5 & 89 & 11 \\
46 & 7 & 10
 \end{ytableau}
\end{center}
\end{example}

Given any weak set-valued tableau $T$, we determine $\mathcal{C}(T)$ by first standardizing tableau $T$. To standardize $T$, first find all $c_1$ occurences of $1$ in $T$ and 
replace them from southwest to northeast with $1,2,\ldots,c_1$. Next, replace the $c_2$ $2$'s from southwest to northeast with $c_1+1,c_1+2,\ldots,c_1+c_2$. Continue this process, replacing
the $c_i$ i's from southwest to northeast with the next available consecutive integer. The resulting tableau is the standardization of $T$. We may then find the 
descent set of the standardization and let $\mathcal{C}(T)$ be the associated composition.

%IS THERE A BETTER (SHORTER) WAY TO EXPLAIN THIS?

\begin{example}
The standardization of the weak set-valued tableau shown below is the tableau $T$ of Example \ref{ex:tabdescent}. Thus $\mathcal{C}(T')=\mathcal{C}(T)=(3,2,4,2)$.
\begin{center}
 $T'=$
\begin{ytableau}
  122 & 3 & 56 & 8 \\
34 & 4 & 7
 \end{ytableau}
\end{center}
\end{example}

Recall from Section \ref{sec:rec} that $Q(h)$ denotes the recording tableau of Hecke insertion.

\begin{theorem}\label{cor:descent}
 Let $h$ be a word and $Q(h)$ be its recording tableau.
 Then $\mathcal{C}(h)=\mathcal{C}(Q(h))$.
\end{theorem}

\begin{proof}
According to Lemma \ref{lem:below}, there is a descent at position $i$ of word $h$ if and only if the entry $i+1$ is strictly
below entry $i$ in $Q(h)$.
\end{proof}

\begin{example}
 Consider $h=13324535$ with $P(h)$ and $Q(h)$ shown below.
\begin{center}
$P(h)=$
 \begin{ytableau}
  1 & 2 & 3 & 5 \\
3 & 4 
 \end{ytableau}\hspace{1in}
$Q(h)=$
\begin{ytableau}
 1 & 23 & 5 & 68 \\
4 & 7
\end{ytableau}
\end{center}
One easily checks that $\mathcal{C}(h)=(3,3,2)=\mathcal{C}(Q(h))$.
\end{example}

\subsection{Decomposition into fundamental quasisymmetric functions}

\begin{theorem} \label{thm:Pp}
For any fixed increasing tableau $T$ of shape $\lambda$ we have  
$$J_\lambda=\displaystyle\sum_{P(h)=T} L_{\mathcal{C}(h)}.$$
\end{theorem}
\begin{proof}
We give an explicit weight-preserving bijection between the set of weak set-valued tableaux of shape $\lambda$ and the 
set of pairs $(h,\sigma')$ where $h=h_1h_2\ldots h_l$ is a word with $P(h)=T$ and $\sigma'$ is a sequence of positive integers $(i_1\leq i_2\leq\ldots\leq i_l)$, where
$i_j<i_{j+1}$ if $j\in \mathcal{D}(h)$.  

Suppose we have a weak set-valued tableau $W$ of shape $\lambda$. To obtain $h$, first standardize $W$. Next,
using $T$ as $P(h)$ and the standardization of $W$ as $Q(h)$, perform reverse Hecke insertion. 
Let the entries of $W$ in increasing order be $i_1,i_2,i_3,\ldots, i_l$, where each $i_j$ is a positive integer, and 
denote $\sigma'=(i_1,i_2,\ldots,i_l)$. We then have $i_j\leq i_{j+1}$ for all $j\in \{1,2,\ldots,l-1\}$, and $i_j<i_{j+1}$ if $j\in\mathcal{D}(h)$
by Theorem \ref{cor:descent}.

For the reverse map, suppose we have $h=h_1h_2\ldots h_l$ with $P(h)=T$ and some $$\sigma'=(i_1\leq i_2\leq \ldots\leq i_l).$$ Then let $W$ be the recording tableau of the insertion of $h$, which uses
 the positive integers of $\sigma'$, i.e. $i_j$ is used to label the special corner $c$ of the insertion 
$P(h_1h_2\ldots h_{j-1}){\overset{H}{\longleftarrow}}h_j$. According to Lemma \ref{lem:below}, the result is a valid weak set-valued tableau. Using Theorem \ref{thm:bij} we 
conclude we indeed have a bijection. 

It remains to note that for a fixed $h$, we have  $$\sum_{\sigma'=(i_1\leq i_2\leq \ldots\leq i_l)} \prod_{j=1}^l x_{i_j} =  L_{\mathcal{C}(h)},$$
where the sum is over $\sigma'$ such that $i_j<i_{j+1}$ if $j\in \mathcal{D}(h)$.
\end{proof}

\begin{example}
 Suppose we start with increasing tableau $T$ and the weak set-valued tableau $W$ shown below.
\begin{center}
$T=$
\begin{ytableau}
 1 & 2 & 3 \\
4 & 5
\end{ytableau}\hspace{1in}
$W=$
 \begin{ytableau}
  122 & 235 & 58 \\
45 & 677
 \end{ytableau}
\end{center}
Performing reverse Hecke insertion with the standardization of $W$ recording the special box $c$ at each step, we obtain $$h=1114412252233.$$
The corresponding composition is $(5,4,4)$, and $$\sigma'=(1,2,2,2,3,4,5,5,5,6,7,7,8).$$

To understand the inverse map, simply let $T=P(h)$ and record the special box $c$ at each step with the postitive integers
of $\sigma'$ to obtain the corresponding weak set-valued tableau.
\end{example}

\begin{remark}
 Pairs $(h,\sigma')$ as above are an analogue of {\it {biwords}} of \cite{LTh}.
\end{remark}

\begin{remark}
 The decomposition of weak stable Grothendieck polynomials $J_{\lambda}$ into fundamental quasisymmetric functions is similar to Stanley's theory of $P$-partitions, see \cite{EC2}.
A different $K$-theoretic analog of such decomposition appears in \cite{LP}.
\end{remark}

\subsection{Map from the $KPR$ to symmetric functions}

Consider the map $\phi: KPR \longrightarrow Sym$ given by 
$$[[h]] \mapsto \sum_{w \equiv h} L_{\mathcal{C}(w)}.$$
\begin{theorem} \label{thm:phi}
 Map $\phi$ is a bialgebra morphism. 
\end{theorem}

\begin{proof}
First we show the map preserves the product. Note that $$L_{\mathcal{C}(w')} \cdot L_{\mathcal{C}(w'')}=\sum_{w\in\text{ Sh}(w',w''[n])}L_{\mathcal{C}(w)},$$ where the sum is over 
all shuffles of $w'$ and $w''[n]$ (see \cite{LP}). Thus, 
\begin{eqnarray}
\phi([[h_1]]\cdot [[h_2]])&=&\phi\left(\sum_{w'\equiv h_1,w''\equiv h_2}w'\shuffle w''\right) \nonumber\\
&=&\sum_{w'\equiv h_1,w''\equiv h_2}L_{\mathcal{C}(w'\shuffle w'')} \nonumber \\
&=&\left(\sum_{w'\equiv h_1}L_{\mathcal{C}(w')}\right) \left(\sum_{w''\equiv h_2} L_{\mathcal{C}(w'')}\right) = \phi([[h_1]])\phi([[h_2]]).\nonumber
\end{eqnarray}

For the coproduct, we show the result 
for $\phi$ applied to $w\equiv h$ with $\phi(w)=L_{\mathcal{C}(w)}$ using that $\Delta(L_{\mathcal{C}(w)})=\sum L_{\beta}\otimes L_\gamma$, where we sum over all $\beta=(\beta_1,\ldots,\beta_k)$ and $\gamma=(\gamma_1,\ldots,\gamma_n)$ such that 
$(\beta_1,\ldots,\beta_k,\gamma_1,\ldots,\gamma_n)=\mathcal{C}(w)$ or $(\beta_1,\ldots,\beta_{k-1},\beta_k+\gamma_1,\gamma_2,\ldots,\gamma_n)=\mathcal{C}(w)$ (see \cite{LP}). 
Then, for any $w\equiv h$, $$\Delta(\phi(w))=\Delta(L_{\mathcal{C}(w)})=\sum L_{\beta}\otimes L_\gamma=\phi\otimes\phi(\Delta(w))$$ 
because the terms in $\Delta(w)=\sum_{i=0}^{|w|}w_1\cdots w_i\otimes w_{i+1}\cdots w_{|w|}$
with $i \in \mathcal{D}(w)$ give exactly the terms $L_\beta\otimes L_\gamma$ where $(\beta_1,\ldots,\beta_k,\gamma_1,\ldots,\gamma_n)=\mathcal{C}(w)$ and all terms with $i\not\in \mathcal{D}(w)$ give 
exactly the terms $L_\beta\otimes L_\alpha$ where $(\beta_1,\ldots,\beta_{k-1},\beta_k+\gamma_1,\gamma_2,\ldots,\gamma_n)=\mathcal{C}(w)$. Thus $\Delta(\phi([[h]]))=\phi\otimes\phi(\Delta([[h]])$.
\end{proof}

\begin{theorem}\label{thm:phi(h)}
 We have $$\phi([[h]]) = \sum_{\r(T) \equiv h} J_{\lambda(T)},$$
where the sum is over all tableaux $K$-Knuth equivalent to $[[h]]$, and $\lambda(T)$ denotes the shape of $T$.
\end{theorem}

\begin{proof}
Using Theorem \ref{thm:Pp}, we have
$$\phi([[h]])=\sum_{w\equiv h}L_{\mathcal{C}(w)}=\sum_{T\equiv P(h)}\sum_{P(w)=T}L_{\mathcal{C}(w)}=
\sum_{T\equiv P(h)}J_{\lambda(T)}=\sum_{\r(T) \equiv h} J_{\lambda(T)}.$$
\end{proof}

\section{Littlewood-Richardson rule}

\subsection{LR rule for Grothendieck polynomials}

\begin{theorem} \label{thm:LR}
 Let $T$ be a URT of shape $\mu$. Then the coefficient $c_{\lambda, \mu}^{\nu}$ in the decomposition 
$$G_{\lambda} G_{\mu} = \sum_{\nu} (-1)^{|\nu|-|\lambda|-|\mu|} c_{\lambda, \mu}^{\nu} G_{\nu}$$ is equal to the number of increasing tableaux $R$
of skew shape $\nu/\lambda$ such that $P(\r(R))=T$.
\end{theorem}

\begin{proof}
In addition to the URT $T$ of shape $\mu$, fix a URT $T'$ of shape $\lambda$. Then by Theorems \ref{thm:ins}, \ref{thm:phi(h)}, and \ref{thm:phi}, 
\begin{eqnarray}
J_\lambda J_\mu&=&\phi([[\r(T')]])\phi([[\r(T)]])\nonumber \\
&=& \phi\left([[\r(T')]]\cdot[[\r(T)]]\right)\nonumber  \\
&=&\phi(\sum_{Y\in T(T'\shuffle T)}\sum_{P(w)=Y}w) \nonumber \\
&=& \sum_{Y\in T(T' \shuffle T)} J_{\lambda(Y)},\nonumber 
\end{eqnarray}
where $T(T'\shuffle T)$ is the finite set of tableaux $Y$ such that $T|_{[|\lambda|]}=T'$ and $P(\r(Y)|_{[|\lambda|+1,|\lambda|+|\mu|]})=T$. 
Thus the coefficient of $J_\nu$ in the product is the number of increasing
tableaux $R$ of skew shape $\nu/\lambda$ such that $P(\r(R))=T$. The desired result follows from Corollary \ref{cor:coincide}.
\end{proof}

%\begin{example}
%Let $\mu=(2,1)$, $T$ be the tableau of shape $\mu$ with $\r(T)=312$, and $\lambda=(1)$. Then 
%$$G_\lambda G_\mu = -G_{(2,2,1)}+G_{(2,1,1)}+G_{(2,2)}-G_{(3,2)}+G_{(3,1)}+G_{(3,2,1)}-G_{(3,1,1)}$$ corresponding to the following fillings of $\nu/\lambda=\nu/(1)$.
%\begin{center}
% \begin{ytableau}
%  \none & 2 \\
%1 & 3 \\
%3
% \end{ytableau}\hspace{.2in}
%\begin{ytableau}
% \none & 2 \\
%1 \\
%3
%\end{ytableau}\hspace{.2in}
%\begin{ytableau}
% \none & 2 \\
%1 & 3 
%\end{ytableau}\hspace{.2in}
%\begin{ytableau}
% \none & 1 & 2 \\
%1 & 3 
%\end{ytableau}\hspace{.2in}
%\begin{ytableau}
% \none & 1 & 2 \\
%3
%\end{ytableau}\hspace{.2in}
% \begin{ytableau}
%  \none & 1 & 2 \\
%1 & 3 \\
%3
% \end{ytableau}\hspace{.2in}
%\begin{ytableau}
% \none & 1 & 2 \\
%1 \\
%3
%\end{ytableau}
%\end{center}
%\end{example}

\begin{example}
The coefficient of $G_{(4,3,1)}$ in $G_{(3,1)} G_{(2,1)}$ is $-3$. To see this, fix $T$ to be the tableau with reading word $312$ as in the previous 
example, note $(-1)^{|\nu|-|\lambda|-|\mu|}=-1$, and notice the tableaux shown below are the only tableaux of shape $(4,3,1)/(3,1)$ with $P(\r(R))=T$.
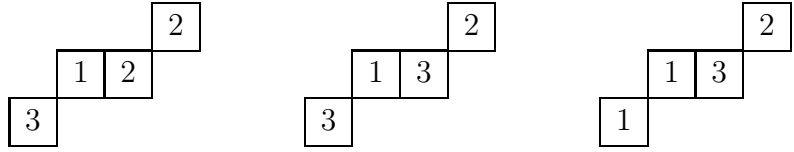
\begin{center}
 \begin{ytableau}
  \none & \none & \none & 2 \\
\none & 1 & 2\\
3
 \end{ytableau}\hspace{.5in}
 \begin{ytableau}
  \none & \none & \none & 2 \\
\none & 1 & 3 \\
3
 \end{ytableau}\hspace{.5in}
\begin{ytableau}
  \none & \none & \none & 2 \\
\none & 1 & 3 \\
1
 \end{ytableau}
\end{center}
\end{example}

Note that the claim may be false if $T$ is not a URT.

\begin{example}
 Suppose we want to find the coefficient of $(4,2)$ in the product of $(4,3,2)$ and $(2,1)$. 
Using Buch's rule \cite{B}, we compute that the coefficient is $3$, corresponding to the following fillings of $(2,1)$:
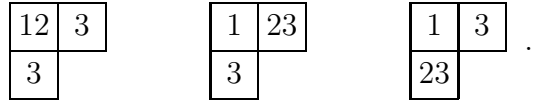
\begin{center}
 \begin{ytableau}
12 & 3 \\
3
 \end{ytableau}\hspace{.5in}
 \begin{ytableau}
1 & 23 \\
3
 \end{ytableau}\hspace{.5in}
 \begin{ytableau}
1 & 3 \\
23
 \end{ytableau}\hspace{.1in}.
\end{center}
However, if we choose the filling of $(3,2)$ with row reading word $34124$, one can easily 
check that there are only two ways to fill $(4,3,2)/(2,1)$ with words equivalent to $34124$ 
that insert into the chosen filling of $(3,2)$. 
The fillings are shown below.
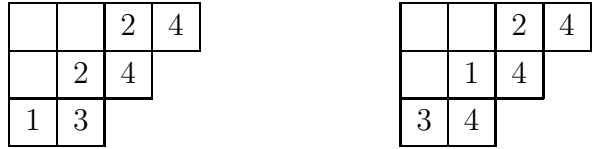
\begin{center}
 \begin{ytableau}
  $ $ & $ $ & 2 & 4 \\
$ $ & 2 & 4 \\
1 & 3
 \end{ytableau}\hspace{1in}
 \begin{ytableau}
  $ $ & $ $ & 2 & 4 \\
$ $ & 1 & 4 \\
3 & 4
 \end{ytableau}

\end{center}
\end{example}

Now we can give our own proof of Theorem \ref{thm:prodfin}.

\begin{proof}
Combine Theorem \ref{thm:phi} with Corollary \ref{cor:prod}.

Alternatively, the argument can be made directly from Theorem \ref{thm:LR}. 
Note that the set of shapes $\nu$ such that there exists an increasing tableau $R$
of skew shape $\nu/\lambda$ such that $P(\r(R))=T$ is finite. This is because each cell in  $\nu/\lambda$ can be filled only with letters occuring in 
$T$, and thus size of each row and column in $\nu/\lambda$ is bounded.
\end{proof}

\subsection{Dual LR rule for Grothendieck polynomials}

Given two Young diagrams, $\lambda$ and $\mu$, define skew shape $\lambda \oplus \mu$ to be the skew shape obtained by putting $\lambda$ and $\mu$ corner to corner. For example, 
The figure below shows $(3,1)\oplus (2,2)$.
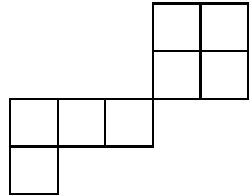
\begin{center}
 \begin{ytableau}
  \none & \none & \none & $ $ & $ $ \\
 \none & \none & \none & $ $ & $ $ \\
$ $ & $ $ & $ $ \\
$ $
 \end{ytableau}
\end{center}

\begin{theorem}\label{thm:skewLR}
 Let $T_0$ be a URT of shape $\nu$. Then the coefficient $d_{\lambda, \mu}^{\nu}$ in the decomposition 
$$\Delta(G_{\nu}) = \sum_{\lambda, \mu} (-1)^{|\nu|-|\lambda|-|\mu|} d_{\lambda, \mu}^{\nu} G_{\lambda} \otimes G_{\mu}$$ is equal to the number of increasing tableaux $R$
of skew shape $\lambda \oplus \mu$ such that $P(\r(R))=T_0$.
\end{theorem}

\begin{proof}
 We have that 
\begin{eqnarray}
 \Delta(J_\nu) &=& \Delta(\phi([[\r(T_0)]])) \nonumber \\
&=& \phi\otimes\phi(\Delta([[\r(T_0)]])) \nonumber \\
&=& \phi\otimes \phi\left(\sum_{(T',T'')\in T(T_0)} \sum_{P(w)=\overline{T'}}w \otimes \sum_{P(w)=\overline{T''}}w\right) \nonumber \\
&=& \sum_{(T',T'')\in T(T_0)}\phi([[\r(\overline{T'})]])\otimes \phi([[\r(\overline{T''})]]) \nonumber \\
&=& \sum_{(T',T'')\in T(T_0)}J_{\lambda(T')}\otimes J_{\lambda(T'')}, \nonumber 
\end{eqnarray}
where $T(T_0)$ is the finite set of pairs of tableaux $T',T''$ such that $P(\r(T') \r(T'')) = T_0$. Letting $R=T'\oplus T''$, the coefficient of $J_\lambda \otimes J_\mu$ is exactly 
the number of increasing tableaux $R$
of skew shape $\lambda \oplus \mu$ such that $P(\r(R))=T_0$. The desired result follows from Corollary \ref{cor:coincide}.
\end{proof}

%\begin{example}
%Fix $T$ to be the increasing tableau of shape $(1,1)$ with $\r(T)=21$. We have that 
%$$\Delta(G_{(1,1)})=G_{(1,1)}\otimes \emptyset + G_{(1)}\otimes G_{(1)} - G_{(1)}\otimes G_{(1,1)}-G_{(1,1)}\otimes G_{(1)}
%+\emptyset\otimes G_{(1,1)}.$$
%Here, the coefficient of $G_{(1)}\otimes G_{(1,1)}$ is $-1$ because of the tableau of shape $(1)\oplus (1,1)$ shown below and $(-1)^{2-1-2}=-1$.
%\begin{center}
% \begin{ytableau}
%  \none & 1 \\
%\none & 2 \\
%2
% \end{ytableau}

%\end{center}

%\end{example}

\begin{example}
Fix $T_0$ to be the URT of shape $(3,2)$ with reading word $45123$. The coefficient of $G_{(2,1)}\otimes G_{(2,1)}$ in $G_{(3,2)}$ is $-3$ because of the following 
three tableaux of shape $(2,1)\oplus (2,1)$.
\begin{center}
 \begin{ytableau}
  \none & \none & 2 & 3 \\
\none & \none & 5 \\
1 & 2 \\
4
 \end{ytableau}\hspace{.5in}
 \begin{ytableau}
  \none & \none & 2 & 3 \\
\none & \none & 5 \\
1 & 4 \\
4
 \end{ytableau}\hspace{.5in}
 \begin{ytableau}
  \none & \none & 2 & 3 \\
\none & \none & 5 \\
1 & 5 \\
4
 \end{ytableau}
\end{center}
\end{example}

Note that the claim may be false if $T_0$ is not a URT.

\begin{example} \label{ex:ThYfail}
Suppose we have \\
\begin{center}
$T_0=$
 \begin{ytableau}
  1 & 2 & 4 \\
3 & 4
 \end{ytableau}\hspace{.1in}.
\end{center}
We saw in Example \ref{ex:3124} that $T_0$ is not a URT. Now let $\lambda=(2,1)$ and $\mu=(3,1)$. According to Buch's rule in \cite{B}, the coefficient of $G_\lambda \otimes G_\mu$ in $\Delta(G_{(3,2)})$ 
is at least 1 due to the following set-valued tableau:
\begin{center}
 \begin{ytableau}
  1 & 1 & 3 \\
23 & 34
 \end{ytableau}\hspace{.1in}.
\end{center}
However, one can check that there is no skew tableau $R$ of shape $(2,1)\oplus (3,1)$ with $P(\r(R))= T_0$.
\end{example}

Now we can give our own proof of Theorem \ref{thm:coprodfin}.

\begin{proof}
Combine Theorem \ref{thm:phi} with Corollary \ref{cor:coprod}.

Alternatively, the argument can be made directly from Theorem \ref{thm:skewLR}. 
Note that the number of pairs $\lambda, \mu$ such that there exists an increasing tableaux $R$
of skew shape $\lambda \oplus \mu$ such that $P(\r(R))=T_0$ is finite.
This is because each $\lambda$ and $\mu$ has to be filled with alphabet of $T_0$ only, hence we can apply Lemma \ref{lem:fin}. 
\end{proof}

%\newpage


\begin{thebibliography}{99}

\normalsize

\bibitem{B}
A.~Buch, A Littlewood Richardson rule for the K-theory of Grassmannians,
\textsl{Acta Mathematica}\, Vol. 189 (2002), pp. 37-78.

\bibitem{BKSTY}
A.~Buch, A.~Kresch, M.~Shimozono, H.~Tamvakis, and A.~Yong, 
Stable Grothendieck polynomials and K-theoretic factor sequences, 
\textsl{Math. Annalen}\ \textbf{2} , Vol. 340, (2008), pp. 359-382. 

\bibitem{BS}
A.~Buch and M.~Samuel,
$K$-theory of miniscule varieties, preprint, arXiv:1306.5419v1.

\bibitem{DHNT}
G.~Duchamp, F.~Hivert, J.-C.~Novelli, J.-Y.~Thibon,
Noncommutative Symmetric Functions VII: Free Quasi-Symmetric Functions Revisited 
\textsl{Annals of Combinatorics}\ \textbf{15}, (2011), pp. 655-673.

\bibitem{FG}
S.~Fomin, C.~Greene, 
Noncommutative schur functions and their applications, 
\textsl{Discrete Mathematics}\ \textbf{193} (1998), pp. 179-200.

\bibitem{FK}
S.~Fomin, A.~Kirillov, 
 Grothendieck polynomials and the Yang-Baxter equation, 
\textsl{Proc. 6th Intern. Conf. on Formal Power Series and Algebraic Combinatorics}, DIMACS, (1994), pp. 183-190.

\bibitem{GK}
C.~Greene, Some partitions associated with a partially ordered set,
\textsl{J. Combinatorial Theory Ser. A}\ \textbf{1}, Vol. 20, (1976), pp. 69-79.

\bibitem{LP}
T.~Lam and P.~Pylyavskyy, 
Combinatorial Hopf algebras and K-homology of Grassmannians,
\textsl{Int. Math. Res. Not.}\ Vol. 2007, (2007), rnm 125.

\bibitem{LS}
A.~Lascoux and M.-P.~Schutzenberger, 
Symmetry and flag manifolds, 
\textit{Lecture Notes in Mathematics}\ Vol. 996, (1983), pp. 118-144.

\bibitem{LTh}
B.~Leclerc and J.Y.~Thibon, 
The Plactic Monoid, in M. Lothaire, Algebraic Combinatorics on Words, 
Cambridge, (2002), pp. 164-195.

\bibitem{M}
L.~Manivel, 
Symmetric Functions, Schubert Polynomials, and Degeneracy Loci, 
American Mathematical Society, (2001).

\bibitem{MS}
E.~Miller, B.~Sturmfels, 
Combinatorial commutative algebra, 
Springer-Verlag, New York, (2005).

\bibitem{PR}
S.~Poirier and C.~Reutenauer,Hopf algebras of tableaux,
\textsl{Ann. Sci. Math. Quebec}\ \textbf{19}, (1995), pp. 79-90.

\bibitem{RT}
V.~Reiner, M.~Taskin,
 The weak and Kazhdan-Lusztig orders on standard Young tableaux, 
\textsl{Formal Power Series and Algebraic Combinatorics}, Vancouver, (2004).

\bibitem{ShS}
S.~Shnider and Sh.~Sternberg, 
Quantum groups: from coalgebras to Drinfeld algebras: a guided tour,
Cambridge, (1993).

\bibitem{EC2}
R.~Stanley, 
Enumerative Combinatorics, Vol 2, Cambridge, (1999).

\bibitem{ThY}
H.~Thomas and A.~Yong, 
Longest Increasing Subsequences, Plancherel-Type Measure and the Hecke Insertion Algorithm,
\textsl{Advances in Applied Math.}\ \textbf{1-4}, Vol. 46, (2011), pp. 610-642. 

\bibitem{ThY2}
H.~Thomas and A.~Yong, 
A jeu de taquin theory for increasing tableaux, with applications to K-theoretic Schubert calculus,
\textsl{Algebra Number Theory}\ \textbf{3}, (2009), pp. 121--148. 

\bibitem{ThY3}
H.~Thomas and A.~Yong, 
The direct sum map on Grassmannians and jeu de taquin for increasing tableaux,
\textsl{Int. Math. Res. Not.}\ \textbf{12}, (2011), pp. 2766--2793.

\bibitem{Zel}
A.~Zelevinsky, 
Representations of finite classical groups: a Hopf algebra approach,
Lecture Notes in Mathematics, Vol 869, Berlin-New York, (1981). 
\end{thebibliography}
\end{document}